\documentclass[11pt]{article}
\usepackage[margin=0.995in]{geometry}
\usepackage{booktabs}

\usepackage{xfrac}  

\usepackage{natbib}
 \bibpunct[, ]{(}{)}{,}{a}{}{,}%

\usepackage{float}
\newfloat{algorithm}{t}{lop}
\floatname{algorithm}{Algorithm}

\usepackage{setspace}

\providecommand{\keywords}[1]
{
  \small	
  \textbf{\textit{Keywords---}} #1
}

\usepackage[colorlinks,citecolor=blue,urlcolor=blue,linkcolor=blue]{hyperref} 

\usepackage{mathrsfs}  
\usepackage{dsfont} 

\usepackage{amsmath}
\usepackage{amssymb}
\usepackage{amsthm}
\usepackage{thmtools}

\newtheorem{theorem}{Theorem}
\newtheorem{lemma}{Lemma}
\newtheorem{claim}{Claim}
\newtheorem{proposition}{Proposition}
\newtheorem{corollary}{Corollary}

\theoremstyle{definition}%
\newtheorem{question}{Question}%
\newtheorem{example}{Example}
\newtheorem{assumption}{Assumption}%

\usepackage{afterpage}

\usepackage{color}              

\usepackage{mathtools}
\usepackage{tikz}
\usetikzlibrary{graphs,graphs.standard}
\usetikzlibrary{arrows}
\usetikzlibrary{arrows.meta}
\usetikzlibrary{graphs,graphs.standard,quotes}
\usetikzlibrary{shapes.geometric}
   \usetikzlibrary{math}
   \usetikzlibrary{shapes.misc}
\usetikzlibrary{shapes.multipart,positioning,patterns,backgrounds}

   \usepackage[most]{tcolorbox}
\usetikzlibrary{patterns}

\let\R\Real

 \DeclareMathOperator*{\argmax}{arg\,max}
\DeclareMathOperator*{\argmin}{arg\,min}

\newcommand\redsout{\bgroup\markoverwith{\textcolor{red}{\rule[0.5ex]{2pt}{0.4pt}}}\ULon}
\newcommand\gsout{\bgroup\markoverwith{\textcolor{green}{\rule[0.5ex]{2pt}{0.4pt}}}\ULon}

\usepackage{array}
\newcommand{\PreserveBackslash}[1]{\let\temp=\\#1\let\\=\temp}
\newcolumntype{C}[1]{>{\PreserveBackslash\centering}p{#1}}
\newcolumntype{R}[1]{>{\PreserveBackslash\raggedleft}p{#1}}
\newcolumntype{L}[1]{>{\PreserveBackslash\raggedright}p{#1}}

\RequirePackage[shortlabels]{enumitem}
\usepackage{verbatim}
\AtBeginDocument{\pagenumbering{arabic}}

\newcommand{\bb}{{\bf b}}
\newcommand{\bc}{{\bf c}}
\newcommand{\bd}{{\bf d}}

\newcommand{\bg}{{\bf g}}

\newcommand{\bx}{{\bf x}}
\newcommand{\by}{{\bf y}}
\newcommand{\bz}{{\bf z}}
\newcommand{\bba}{{\bf A}}

\newcommand{\bbd}{{\bf D}}

\newcommand{\bbl}{{\bf L}}

\newcommand{\bbu}{{\bf U}}

\newcommand{\bgamma}{{\boldsymbol{\gamma}}}

\newcommand{\param}{{\boldsymbol{\theta}}}

\newenvironment{itemize*}%
  {\begin{itemize}%
    \setlength{\itemsep}{0.5em}%
    \setlength{\parskip}{0pt}}%
  {\end{itemize}}

\newcommand{\btheta}{{\boldsymbol{\theta}}}

\newcommand{\bpsi}{{\boldsymbol{\psi}}}

\newcommand{\Prb}{\mathbb{P}}  
\newcommand{\Exp}{\mathbb{E}}

\let\R\Real

\def\conv{\operatorname{conv}}	
	
\newcommand{\bone}{{\boldsymbol{1}}}
\newcommand{\bzero}{{\boldsymbol{0}}}

\usepackage{graphicx}
\usepackage{bbm}
\usepackage{ifthen}
   \usetikzlibrary{math}
   \usetikzlibrary{shapes.misc}
   \usetikzlibrary{shapes.multipart,positioning,patterns,backgrounds}
\newcommand{\ubar}[1]{\text{\b{$#1$}}}

\usepackage{multirow}

\usepackage{accents}
    \DeclareMathSymbol{\widetildesym}{\mathord}{largesymbols}{"65}

\title{The Value of  Human Expertise}

\author{Bradley Sturt
}
\date{August 23, 2026}

\begin{document}

\singlespacing

\maketitle

\begin{abstract}
We consider optimization applications with unknown parameters where the decision maker believes that the optimal value of the nominal problem—the optimization problem they would have solved if the true parameters were known—is unlikely to be large. This belief derives from information that humans have that is not captured in datasets, obtained from domain knowledge and interacting with the physical world. We propose an approach to evaluating policies that provides tighter performance guarantees if the decision maker’s belief happens to be correct. Our main result shows that if computing a policy’s worst-case performance is a convex program, then the \emph{value of human expertise}---the maximum improvement in performance guarantees that can be obtained from the belief about the nominal problem---is equal to the minimax gap of a max-min problem. We illustrate our developments in assortment optimization and shortest path problems.\looseness=-1
\end{abstract}

\setlength{\parskip}{0.5em}


\section{Introduction} \label{sec:intro}

Consider a parent company that seeks to optimize product assortments at individual brick-and-mortar stores. The parent company has access to the transactional sales data generated by a store's past assortments. However, because the number and variety of the store's past assortments are limited, the parent company has insufficient data to estimate an accurate store-specific discrete choice model or identify an assortment that can be trusted to outperform the store's best past assortment. That said, the past assortments were not chosen \emph{randomly}: rather, the store manager chose the past assortments using private information about local consumer preferences obtained from observing and interacting with customers that visit the store. As a result, the parent company believes it is unlikely that the store's best past assortment was {highly suboptimal}. \emph{Can that belief help the parent company identify a new assortment that outperforms the store's best past assortment?}\looseness=-1

The above example is an instance of a new class of problems that we refer to as \emph{optimization with human expertise}.  These  are optimization applications with unknown parameters where the decision maker believes that the optimal value of the nominal problem---the maximization problem they would have solved if the true parameters were known---is unlikely to be large. This belief can originate from various sources, such as a belief that a past policy was not highly suboptimal. For example, in the setting described above, if the store's best past assortment generated an expected revenue of $\$21$ per customer, then the parent company may believe it is unlikely that the expected revenue of the optimal assortment is much larger than $\$21$. A decision maker's belief about the nominal problem can also derive from other sources such as domain knowledge and intuition.\looseness=-1 %

There are fundamental challenges with trying to  incorporate human beliefs into optimization applications. First, a decision maker's belief may be \emph{incorrect}. For example, in the assortment optimization setting,  it may be the case that the  store's best past assortment was in fact highly suboptimal, perhaps because it was chosen by a store manager with an objective other than maximizing revenue, such as maximizing market share. Second, in practice, a decision maker's belief  often comes in the form of a  \emph{vibe} rather than a numerical assertion. For example, a parent company may believe that it is unlikely that the store's best past assortment was highly suboptimal, but may be hard pressed to transform that belief into a claim that the expected revenue of the store's best past assortment is guaranteed to be within, say, 31.2\% of optimal. Third, even if the decision maker could turn their belief into a numerical bound on the optimal value of the nominal problem, that partial information is generally too crude for identifying the true parameters of the nominal problem and insufficient  for recovering the optimal policy.\looseness=-1

\subsection{Contributions}
The objective of this paper is to propose an approach to incorporating human beliefs that contends with the above challenges, and in doing so,  to shed light on the value that human expertise can provide in this era of algorithm-driven decision making.  The main contributions of this paper are the following. %

\paragraph{Model.}

We formalize optimization with human expertise as follows. A decision maker faces an optimization application defined by a set of feasible policies and an objective function with an unknown parameter.  All of the information from available datasets about the true but unknown parameter has been summarized into an uncertainty set of parameters.\footnote{Our setup (\S\ref{sec:managerialexpertise})   is general and can address applications in which policies and parameters  are infinite-dimensional.  } In addition, the decision maker believes that the optimal value of the nominal problem---the problem they would have solved if the true parameter were known---is unlikely to be large. The decision maker seeks a policy that can be trusted to perform well across the parameters from the uncertainty set, but is guaranteed to perform even better if the optimal value of the nominal problem happens to be small.\looseness=-1

We propose a general approach to {evaluating} and selecting policies in this class of problems. Rather than trying to elicit details about the belief of the decision maker in advance, the approach instead provides a  menu of policies to a decision maker, along with a particular object for each policy that we call a \emph{nominal curve}. In a nutshell, a nominal curve is a set of worst-case performance guarantees for a policy that hold under every  possible upper bound on the optimal value of the nominal problem, or under every possible suboptimality gap for the best past policy.\footnote{For example, in the assortment optimization setting from the beginning of \S\ref{sec:intro}, a nominal curve shows the worst-case expected revenue of a new assortment if the  expected revenue of the store's best past assortment happened to be within X\% of optimal, for every possible value of X.}
By comparing the policies through their nominal curves, the decision maker can obtain an understanding of how the  policies in the menu can be trusted to perform under different scenarios of the optimal value of the nominal problem,  allowing them to apply their own confidence in their belief and their judgment of what scenarios are plausible to select a policy from the menu to use in practice.\looseness=-1 %

Nominal curves may be viewed as attractive from a practical standpoint for several reasons. 
First, they address the fact that the decision maker's belief about the nominal problem may be incorrect by containing worst-case performance guarantees for a policy that hold even if the optimal value of the nominal problem is arbitrarily large. Second, it is shown under mild assumptions that  nominal curves are convex, continuous, and nonincreasing (Proposition~\ref{prop:shape} in \S\ref{sec:managerialexpertise:computation}), which makes them relatively simple to visualize, as we show in \S\S\ref{sec:assortment:results} and \ref{sec:motivatingexample_shortest_path}. Third, a nominal curve does not require the decision maker to transform their belief into a fixed upper bound on the optimal value of the nominal problem, which may be inaccurate or difficult to elicit.\looseness=-1%

\paragraph{Theory.}
For the nominal curve of a policy to be useful to a decision maker, it must show that the worst-case performance of a policy improves when the optimal value of the nominal problem happens to be small. This raises many practical questions: for example, in what settings do nominal curves contain worst-case performance guarantees that are less conservative than those that would have been obtained in the absence of a belief about the optimal value of the nominal problem? %
By how much can the worst-case performance guarantees of a policy  be improved if the decision maker's belief about the nominal problem turns out to be correct?\looseness=-1%

To answer these questions, we introduce a quantity  that we call the \emph{value of human expertise} (\S\ref{sec:vhe}).  The value of human expertise is the maximum difference between the optimal values of two problems. The first is a robust optimization problem, that is, the problem of selecting a policy that performs best under the worst-case parameters from the uncertainty set. The second is a robust optimization problem where the uncertainty set has been augmented with constraints that remove all parameters that, if true, would have led to a nominal problem with an optimal value that exceeds a threshold. %
As such, the value of human expertise captures the maximum improvement in the optimal value of a robust optimization problem that can be obtained from a belief that the optimal value of the nominal problem is unlikely to be large. Stated alternatively, the value of human expertise is large if and only if there exist nominal curves that provide worst-case performance guarantees that are much less conservative than those that can be obtained by solving a robust optimization problem.\looseness=-1

Our main result  (Corollary~\ref{cor:main} in \S\ref{sec:maintheorem}) shows that the value of human expertise has a simple characterization. 
Specifically, we prove that  if the problem of computing the worst-case performance of a policy over the uncertainty set is a convex problem, then the value of human expertise is equal to the gap  between the optimal values of the min-max and max-min formulations of a robust optimization problem. This result, which holds for infinite-dimensional policy and parameter spaces, shows that a decision maker's belief about the nominal problem is valuable precisely in the applications where minimax duality does not hold. %
The proof of our main result  follows from a simple---and, as best we can tell, novel---theorem about max-min problems (Theorem~\ref{thm:fundamental} in \S\ref{sec:maintheorem}). The theorem shows that if a max-min problem has an inner problem that is convex, then a pure strategy of the outer problem  that performs best against all of the optimal solutions of the min-max problem yields an objective  value that is equal to the optimal value of the min-max problem. 

\paragraph{Applications.}
We illustrate our approach in two applications.  First, we consider the hyperlocal assortment optimization setting that was described at the beginning of \S\ref{sec:intro}, in which a parent company seeks to optimize assortments at local brick-and-mortar stores (\S\ref{sec:assortment:example1}). We consider a nonparametric {setting} in which the parent company's goal is to identify  an assortment that performs well across all of the random utility maximization models that are consistent with the transactional sales data generated by the store's past assortments \citep{farias2013nonparametric,sturt2025value}.  In a stylized numerical example, we show that nominal curves make it possible to identify new assortments to recommend to the store that in the worst case do not perform much worse than the store’s best past assortment, but are guaranteed to  strictly outperform the store’s best past assortment if the expected revenue of the store’s best past assortment happens to be close to optimal (Figure~\ref{fig:assortment_small} in \S\ref{sec:assortment:results}).\looseness=-1%
  
Second,  we consider combinatorial optimization problems (e.g., shortest path, bipartite matching, traveling salesman problems) with uncertain cost coefficients and budget uncertainty sets (\S\ref{sec:applications:optimizercurse}). Our analysis of this setting is motivated by situations where decision makers themselves  have domain knowledge. For example, experienced emergency dispatchers may have knowledge about traffic delays that commonly occur when there are flash floods; during a severe storm, the dispatcher may not precisely know the location of flooding, but may believe it is likely that all shortest routes will experience delays. For this setting, we establish several theoretical conditions under which relatively  loose beliefs about the optimal value of the nominal problem can lead to improved performance guarantees. We highlight these theoretical results through numerical experiments on a class of randomly generated shortest path problems from \cite{MelvynSim2003}.\looseness=-1

To facilitate the deployment of our proposed approach in applications, we develop two methods that can be applied in assortment and combinatorial optimization problems for computing a menu of policies that are \emph{pointwise optimal}, meaning that their nominal curves give the best possible performance guarantees in at least one scenario (\S\S\ref{sec:managerialexpertise:generating_policies} and \ref{sec:applications:optimizercurse:algorithms}).  The first method is a compact mixed-integer programming formulation that applies to zero-one network flow problems (such as shortest path and bipartite matching) and cardinality-constrained assortment optimization under the multinomial logit model with polyhedral uncertainty sets (\S\ref{sec:algorithm:reform}).  The second method is a simple cutting plane method that can be applied in more general settings (\S\ref{sec:algorithm:cutting}). We propose other approaches for generating menus of policies in \S\ref{appx:disappointment}.

\subsection{Related Literature} \label{sec:lit}
This paper draws on and contributes to several fields including robust optimization, inverse optimization, assortment optimization, and minimax theory.  

 \paragraph{Robust optimization.}

There is a vast literature on designing uncertainty sets in robust optimization based on probabilistic guarantees, risk aversion, historical data, and machine learning models; see \cite{kuhn2025distributionally,lou2024estimation}, and references therein. Most of our main results hold for general uncertainty or ambiguity sets that are nonempty, compact, and convex (see \S\S\ref{sec:managerialexpertise:computation} and \ref{sec:mainresults}). As such, our main results can be applied concurrently with many existing methods for designing uncertainty and ambiguity sets. Compared to papers such as \cite{iancu2014pareto} that reduce the conservatism of robust optimization for a given uncertainty set, we propose a new approach to evaluating  policies by studying the relationship between the parameters in an uncertainty set and the possible optimal values of the nominal problem.\looseness=-1 %

Our paper relates to literature on reducing the conservatism of uncertainty sets by incorporating information about an auxiliary optimization problem. \cite{long2023robust} and \cite{sim2025analytics} introduce robust satisficing and propose computing the target in that model by solving an empirical optimization problem to contend with the optimizer's curse.  \cite{wang2023learning} tune the parameters of an uncertainty set to perform well across a contextual family of different problems. Closer to the present paper is the work of \cite{bennouna2026data,bennouna2026data2}, who add constraints into uncertainty sets based on linear projections of the linear objective function with uncertain coefficients and analyze when these constraints are sufficient for identifying the optimal solution of the nominal problem. In our work, the auxiliary problem is the nominal problem, i.e., the problem we would have solved if the true parameters were known, and we incorporate  bounds on the optimal value of the nominal problem by adding constraints into an uncertainty set. 

 \paragraph{Private information and inverse optimization.} An example of an application where humans have access to private information that is not available to the decision maker is hyperlocal revenue management, whereby a parent company or distributor seeks to optimize pricing or assortment decisions at local stores. In these settings, there is literature documenting that managers have access to private information about local consumer preferences that influences their decisions \citep{kok2008assortment,farias2017building}, which relates to a much broader literature on the value of private information in human-AI collaboration \citep{kesavan2020field,balakrishnan2026human}. Our paper studies how such private information can be incorporated indirectly through beliefs about the quality of past policies that were selected by humans. %

A related stream of research is the field of inverse optimization, which focuses on using partial information about the optimal solutions of a nominal problem to infer the parameters of its objective function~\citep{ahuja2001inverse,chan2025inverse}. Typical motivating examples for inverse optimization are situations where one observes decisions made by experts (such as routing choices of drivers or medical decisions made by doctors) and wishes to ``reverse engineer" the objective or utility function that the experts were trying to optimize.
Many variants have been proposed, including those where the optimal solutions are data-driven and subject to noise~\citep{aswani2018inverse,mohajerin2018data_inverse} and where information about the optimal value is available~\citep{ahmed2005inverse}.   

Our paper studies a problem setting that lies somewhere between traditional robust optimization and inverse optimization. Similar to robust optimization, we start with an uncertainty set of parameters of the objective function, and like inverse optimization, we additionally have partial information about the nominal problem with the true parameters. However, we are interested in cases where that partial information is too crude for identifying the true parameters of the nominal problem and insufficient  for recovering the optimal policy. One of the takeaways of the numerical examples of this paper (\S\S\ref{sec:assortment:results} and \ref{sec:motivatingexample_shortest_path})  is that intersecting a ``conservative" uncertainty set with ``crude" information about the optimal value of the nominal problem can lead to robust optimization problems that are practically useful.

\paragraph{Assortment optimization.} 
A central problem in the field of assortment optimization is identifying the discrete choice model from aggregated transactional data generated by a store's past assortments.  The challenge is that there are often many stochastically rational (i.e., random utility maximization) choice models that can fit such data perfectly. To contend with this, a stream of literature \citep{rusmevichientong2012robust,farias2013nonparametric,bertsimas2017robust,jin2022distributionally,desir2024robust,wang2024randomized,sturt2025value,ruan2026nonparametric} has focused on constructing  uncertainty sets of discrete choice models and finding an assortment that does best in the worst case by solving a robust assortment optimization problem. This work fits within a broader literature in revenue management and economics on prior-free and non-Bayesian approaches to pricing and mechanism design; see \cite{bergemann2011robust,koccyiugit2020distributionally,anunrojwong2024best,bahamou2024fast}, and references therein.

Our paper contributes to the above revenue management literature in three ways. First, we propose a new robustness criterion %
that is designed for settings where a decision maker believes that the optimal value of the nominal problem is unlikely to be large. %
Second, we provide numerical evidence that uncertainty sets that are constructed using nonparametric modeling techniques from \cite{farias2013nonparametric} can contain random utility maximization models that, if true, would imply that  past assortments were highly suboptimal (see \S\ref{sec:assortment:example1}). Third, we show that eliminating those parameters can induce a phase transition of a robust assortment optimization problem from being overly conservative  to practically useful.  %

\paragraph{Minimax theory.} 

There is a rich history in operations research and game theory of establishing properties of max-min problems that ensure that the minimum and maximum can be interchanged. %
Classic conditions for minimax theorems include \cite{fan1953minimax} and \cite{sion1958general}, and conditions that arise in the context of  robust optimization applications can be found in \cite{nilim2005robust,iyengar2005robust,wei2024adjustability,desir2024robust,zhen2025unified,shafiee2026nash}. Our paper differs from the above literature because we  prove our main results under assumptions that are insufficient for minimax theorems (Assumptions~\ref{ass:main} and \ref{ass:convex} in \S\ref{sec:managerialexpertise:computation}). The study of mixed strategies in the context of stochastic programming and robust optimization includes \cite{delage2019dice,wang2024randomized,guan2026randomized}. Our work focuses on  high-stakes applications where a decision maker is risk-averse and randomized policies are impossible to implement or undesirable to the decision maker; examples of such applications can be found throughout this paper.\looseness=-1%

As best we can tell, the closest related work  to our Theorem~\ref{thm:fundamental}  is \citet[p. 406]{terkelsen1972some}, which presents proof techniques based on finite intersections that are similar to ours, specifically our proof steps that we relegate  to Appendix~\ref{app:topology_application}. While our proof techniques for Theorem~\ref{thm:fundamental} are elementary, we are not aware of a prior statement of our theorem in the literature, and our theorem does not appear to follow immediately from classical results from papers such as \cite{fan1953minimax} and \cite{sion1958general}.  The value of our Theorem~\ref{thm:fundamental} lies in its generality:  it does not require convexity of the outer problem, does not assume that sets of optimal strategies for the inner or outer problems are singletons, and extends to inner and outer problems that are infinite-dimensional.\looseness=-1  %


\section{Optimization with Human Expertise}\label{sec:managerialexpertise}
\subsection{Problem Setting} \label{sec:managerialexpertise:setting}
We consider  optimization problems  of the form 
\begin{align}
    \max_{\bx \in \mathcal{X}} f(\bx,\bar{\param}) \label{prob:true}
\end{align}
where the policies are chosen from   a nonempty compact set $\mathcal{X}$ and where the true parameter of the objective function $\bar{\param}$ is unknown.
To ensure that  optimums are attained, we assume throughout that  $f(\cdot,\cdot)$ is bounded, upper semicontinuous in its first argument, and lower semicontinuous in its second argument.\footnote{Note that these assumptions allow for the  policies $\bx$ and the parameter  $\bar{\btheta}$ to be infinite-dimensional.} The parameter $\bar{\param}$ can correspond, for example,  to a probability measure, a random utility   model, or edge costs in a network.\looseness=-1

We focus on the problem setting in which all of the information from available datasets about the true but unknown parameter $\bar{\btheta}$ has been summarized by a nonempty compact set of parameters $\mathcal{U}$, referred to as an uncertainty set. We assume that the uncertainty set was constructed with the goal of being small while  containing the true parameter $\bar{\btheta} \in \mathcal{U}$. The study of data-driven techniques for constructing uncertainty sets with rigorous guarantees has rapidly evolved into what is now a relatively mature discipline (see \S\ref{sec:lit}). In assortment optimization, for example, uncertainty sets  can be constructed as a set of random utility maximization models   that are consistent with transactional sales data generated by a store's past assortments.\looseness=-1

In this paper, we assume that in addition to having an uncertainty set, the decision maker believes that it is unlikely that the optimal value of the nominal problem~\eqref{prob:true} is large.  The belief is derived from information that is not captured in datasets, obtained from domain knowledge and interacting with the physical world. The decision maker seeks a policy that can be trusted to perform well across the parameters from the uncertainty set, but performs even better if the optimal value of the nominal problem~\eqref{prob:true} happens to be small.\looseness=-1

\subsection{A New Approach to Evaluating Policies} \label{sec:managerialexpertise:nominalcurves}
We propose the following approach to evaluating policies in the above problem setting.   For each $\eta \in \R$,  define the \emph{reduced uncertainty set} corresponding to $\eta$ as \looseness=-1
\begin{align*}
 \mathcal{U}_\eta \coloneqq     \left \{ \param \in \mathcal{U}: \max_{\bx' \in \mathcal{X}} f(\bx',\param) \le \eta \right \} = \left \{ \param \in \mathcal{U}: f(\bx',\param) \le \eta \;\; \forall \bx' \in \mathcal{X} \right \}. 
\end{align*}
In words, the reduced uncertainty set is equal to the original uncertainty set with constraints that exclude all parameters that, if true,  would lead the nominal problem to have an optimal value that exceeds $\eta$.
Instead of evaluating a policy $\bx$ by its worst-case performance, 
 we propose evaluating the policy by its \emph{nominal curve}\looseness=-1%
\begin{align}
\left \{\left( \eta, \min_{\btheta \in \mathcal{U}_\eta} f(\bx,\btheta) \right):     \eta \ge \ubar{\eta} \right \}  \label{line:wc_eta_set}
\end{align} 
where $\ubar{\eta} \coloneqq \min\{\eta: \mathcal{U}_\eta \neq \emptyset \}$ is the smallest scalar for which the reduced uncertainty set is nonempty.\footnote{An explicit formula for $\ubar{\eta}$  is  found in Proposition~\ref{prop:eta} in \S\ref{sec:simplebounds}.} %
The nominal curve provides the usual worst-case performance guarantee for the policy for all sufficiently large $\eta$ and gives performance guarantees that may be less pessimistic if $\bar{\btheta} \in \mathcal{U}$ and the optimal value of the nominal problem~\eqref{prob:true} happens to be less than or equal to $\eta$, for each $\eta \ge \ubar{\eta}$. As we show below, the nominal curve thus offers the decision maker a rigorous way to transform a belief that it is unlikely that the optimal value of the nominal problem~\eqref{prob:true} is large into performance guarantees for a policy that are practically useful.

As a motivating example, it is common in assortment optimization for a parent company to have access to historical sales data generated by a store's past assortments, and for the store manager to be risk-averse and reluctant to experiment with new assortments that might lead to a decline in expected revenue.    Given the historical sales data, the parent company can construct an uncertainty set of random utility maximization models that are consistent with the historical sales data, and then solve a robust optimization problem to identify an assortment that can be trusted to outperform the store's best past assortment across all random utility models in the uncertainty set \citep{farias2013nonparametric,sturt2025value}.  In \S\ref{sec:assortment:results}, we show that such data-driven uncertainty sets can contain random utility maximization models  that, if true, would imply that the store's past assortments  were highly suboptimal.  This would be surprising in many practical settings in which store managers did not choose the past assortments \emph{randomly}, but rather were informed by  private information  about the types of customers that typically visit the store. The nominal curve~\eqref{line:wc_eta_set} makes it possible to, for example, identify new assortments to recommend to the store that in the worst case do not perform much worse than the store’s best past assortment, but are guaranteed to  strictly outperform the store’s best past assortment if the expected revenue of the store’s best past assortment happens to be close to optimal.

The nominal curve~\eqref{line:wc_eta_set} is not a scalar. This is motivated by the fact that the decision maker's belief that it is unlikely that the optimal value of \eqref{prob:true} is large may be subjective, and so it may be difficult for a decision maker to transform their belief into, for example, an accurate upper bound on the optimal value of \eqref{prob:true}. A nominal curve thus allows a decision maker to obtain performance guarantees for a policy under a variety of possible scenarios about the optimal value of the  nominal problem~\eqref{prob:true} that the decision maker believes may be plausible, and to compare policies under those possible scenarios by plotting their nominal curves; see \S\ref{sec:assortment:results} and \S\ref{sec:motivatingexample_shortest_path}. By capturing the worst-case performance of a policy over the entire uncertainty set in the case of sufficiently large $\eta$\footnote{It follows from the boundedness of $f(\cdot,\cdot)$ that $\mathcal{U}_\eta = \mathcal{U}$ for all sufficiently large $\eta$.}, the nominal curve also shows the decision maker how a policy can be trusted to perform even if their belief is incorrect.\looseness=-1

\subsection{Generating Policies} \label{sec:managerialexpertise:generating_policies}
Because \eqref{line:wc_eta_set} is not a scalar,  there may not exist a policy that is \emph{universally optimal} with respect to \eqref{line:wc_eta_set}. Specifically, we observe that the nominal curves~\eqref{line:wc_eta_set} corresponding to feasible policies are pointwise upper bounded by\looseness=-1
\begin{align}
\left \{\left( \eta, \max_{\bx \in \mathcal{X}} \min_{\btheta \in \mathcal{U}_\eta} f(\bx,\btheta) \right): \eta \ge \ubar{\eta}\right \}. \label{prob:robust_eta_set}
\end{align}
If there exists a policy whose nominal curve~\eqref{line:wc_eta_set} coincides with the upper bound \eqref{prob:robust_eta_set} for every $\eta$, then that policy would be considered universally optimal. %
 Because such a policy may not exist, one can generate a menu of policies  to offer to the decision maker. 

A simple approach to generating a menu of policies is to solve the \emph{reduced robust optimization} problem
\begin{align}
\max_{\bx \in \mathcal{X}} \min_{\btheta \in \mathcal{U}_\eta} f(\bx,\btheta) \label{prob:robust_eta}
\end{align}
 across a discrete range of values of $\eta \in [\ubar{\eta},\infty)$. The optimal policies for those reduced robust optimization problems have nominal curves that intersect  the upper bound \eqref{prob:robust_eta_set}; we refer to these as \emph{pointwise optimal} nominal curves.  
  Given the menu of policies obtained by solving \eqref{prob:robust_eta} for different values of $\eta$, the decision maker can compare the nominal curves~\eqref{line:wc_eta_set} of those policies and apply their own judgment about what scenarios are plausible to select a single policy. We apply this approach to generating menus of policies in the numerical examples in  \S\S\ref{sec:assortment:results} and \ref{sec:motivatingexample_shortest_path}, and alternative approaches to generating a menu of policies can be found in \S\ref{appx:disappointment}.\looseness=-1

\subsection{Computation} \label{sec:managerialexpertise:computation}
There are two relevant computational tasks to consider: computing a nominal curve~\eqref{line:wc_eta_set} for a fixed policy, and finding a menu of policies through solving \eqref{prob:robust_eta} for a range of values of $\eta$.  Neither task is trivial in general in light of the minimal assumptions stated at the beginning of \S\ref{sec:managerialexpertise:setting}, which are repeated below:
\begin{assumption}\label{ass:main}
$\mathcal{X},\mathcal{U}$ are compact, nonempty sets, and $f(\cdot,\cdot)$ is a bounded function that is upper semicontinuous in its first argument and lower semicontinuous in its second argument. 
\end{assumption} 
Those two computational tasks simplify, however, if the following assumption is also satisfied:\looseness=-1
\begin{assumption}\label{ass:convex}
$\mathcal{U}$ is a convex set and $\param \mapsto f(\bx,\param)$ is a convex function for all $\bx \in \mathcal{X}$.
\end{assumption} 
The above assumption, which is common in the literature and often satisfied in real-world applications, does not guarantee that nominal curves~\eqref{line:wc_eta_set} are easy to compute, nor does it imply that the reduced robust optimization problem~\eqref{prob:robust_eta} is easy to solve. For example, Assumptions~\ref{ass:main} and \ref{ass:convex} do not preclude $\mathcal{U}$ from being an infinite-dimensional set of parameters. That said, Assumptions~\ref{ass:main} and \ref{ass:convex}  ensure   for each $\bx \in \mathcal{X}$ and $\eta \ge \ubar{\eta}$ that
$$\min_{\btheta \in \mathcal{U}_\eta} f(\bx,\btheta)$$
is a convex optimization problem over a compact convex set $\mathcal{U}_\eta$. As such, the nominal curve~\eqref{line:wc_eta_set} for every fixed $\bx \in \mathcal{X}$ can be calculated to any accuracy by solving multiple convex optimization problems, one for each $\eta$ in a sufficiently fine grid. The above assumptions also ensure that nominal curves~\eqref{line:wc_eta_set} have a simple structure:
\begin{restatable}{proposition}{propone}
\label{prop:shape}
If Assumptions~\ref{ass:main} and \ref{ass:convex} hold and $\bx \in \mathcal{X}$, then the function $v_\bx(\eta) \coloneqq \min_{\btheta \in \mathcal{U}_\eta} f(\bx,\btheta)$ is nonincreasing, convex, and continuous for $\eta \in [\ubar{\eta},\infty)$. 
\end{restatable}
The above proposition implies that nominal curves~\eqref{line:wc_eta_set} for fixed policies as well as the upper bound curve~\eqref{prob:robust_eta_set} have a structure that is relatively easy to visualize, as we show in \S\S\ref{sec:assortment:results} and \ref{sec:motivatingexample_shortest_path}.  Proposition~\ref{prop:shape} also implies that if a policy $\bx$ satisfies $\min_{\btheta \in \mathcal{U}_\eta} f(\bx,\btheta) > \min_{\btheta \in \mathcal{U}} f(\bx,\btheta)$ for some $\eta$, and if $\bar{\btheta} \in \mathcal{U}$, then the smaller the optimal value of the nominal problem~\eqref{prob:true} happens to be, the better the policy is guaranteed to perform in the worst case.  %
The proof of  Proposition~\ref{prop:shape}, which is found in Appendix~\ref{appx:misc_eta_nominal}, makes use of the closed-form expression of $\ubar{\eta} \coloneqq \min\{\eta: \mathcal{U}_\eta \neq \emptyset \}$ that is established in Proposition~\ref{prop:eta} in  \S\ref{sec:simplebounds}. 

\subsection{The Value of Human Expertise} \label{sec:vhe}

To evaluate the capacity of nominal curves to provide performance guarantees that are practically useful, 
we introduce a quantity that we call the \emph{value of human expertise}. This quantity, defined below, is the maximum improvement in the optimal value of a robust optimization problem that can be achieved by eliminating parameters from the uncertainty set that, if true, would have led the nominal problem to have a large optimal value. Equivalently, it is the difference between the maximum and minimum of the upper bound curve~\eqref{prob:robust_eta_set}:\looseness=-1
\begin{align}
\Delta &\coloneqq \max \limits_{ \eta \ge \ubar{\eta}} \left \{\max \limits_{\bx \in \mathcal{X}} \min \limits_{\btheta \in \mathcal{U}_\eta} f(\bx,\btheta) \right \}  - \min \limits_{\eta \ge \ubar{\eta}} \left \{  \max \limits_{\bx \in \mathcal{X}} \min \limits_{\btheta \in \mathcal{U}_\eta} f(\bx,\btheta) \right \} \notag \\
&= \max \limits_{\bx \in \mathcal{X}} \min \limits_{\btheta \in \mathcal{U}_{\ubar{\eta}}} f(\bx,\btheta) - \max \limits_{\bx \in \mathcal{X}} \min \limits_{\btheta \in \mathcal{U}} f(\bx,\btheta) \notag %
\end{align}
The second equality holds because $\mathcal{U}_{\ubar{\eta}}$ is nonempty, $\mathcal{U}_{\ubar{\eta}} \subseteq \mathcal{U}_{\eta}$ for all $\eta \ge \ubar{\eta}$, and $\mathcal{U}_{\eta} = \mathcal{U}$ for all sufficiently large $\eta$ by the boundedness of $f(\cdot,\cdot)$. 
Note that the attainment of the maxima and minima in each of the above optimization problems  follows from the assumptions at the beginning of \S\ref{sec:managerialexpertise:setting}, which were restated as Assumption~\ref{ass:main} in \S\ref{sec:managerialexpertise:computation}.\looseness=-1

The value of human expertise can be interpreted as the maximum improvement in performance guarantees that can be obtained from a belief that the optimal value of \eqref{prob:true} is unlikely to be large. Indeed, if this quantity is equal to zero, then it is not possible for nominal curves to provide performance guarantees that are less conservative than those that would be obtained by solving a robust optimization problem. If the value of human expertise is large, then it guarantees the existence of policies with non-trivial nominal curves, that is, the existence of policies $\bx \in \mathcal{X}$ for which $\min_{\btheta \in \mathcal{U}_\eta} f(\bx,\btheta)$ is much larger than $\min_{\btheta \in \mathcal{U}} f(\bx,\btheta)$ as well as the optimal value of a robust optimization problem for some values of $\eta$. {It follows from \S\ref{sec:managerialexpertise:generating_policies} that such policies can be obtained by solving the reduced robust optimization problem~\eqref{prob:robust_eta} for small values of $\eta$.} In sum, a large $\Delta$ is a necessary and sufficient condition for there to exist nominal curves that offer much  stronger performance guarantees than can be obtained by robust optimization. %

There are many questions related to the value of human expertise that are relevant from theoretical and practical perspectives.
In what applications can the value of human expertise be strictly positive?  How does the value of human expertise relate to the structure of a traditional robust optimization problem $\max_{\bx \in \mathcal{X}} \min_{\param \in \mathcal{U}} f(\bx,\param)$?
More generally, for there to be a non-zero gap between the optimal values of the reduced robust optimization problem~\eqref{prob:robust_eta} and a traditional robust optimization problem, must  $\eta$ be a tight bound on the optimal value of the nominal problem~\eqref{prob:true}? %
 In addition to shedding light on those questions, our main goal of this paper is to provide answers to the following two questions:\looseness=-1
\begin{question} \label{q1}
What is the maximum possible value of $\Delta$? 
\end{question}
\begin{question} \label{q2}
In what settings does $\Delta$ attain that maximum possible value?
\end{question}



\section{The Value of Human Expertise and the Minimax Gap} \label{sec:mainresults}
In this section, we show that there are simple answers to Questions~\ref{q1} and \ref{q2}. %
 In \S\ref{sec:simplebounds}, we answer Question~\ref{q1} by developing an upper bound on the value of human expertise that holds under Assumption~\ref{ass:main}.  In \S\ref{sec:maintheorem}, we present the main result of this paper, Corollary~\ref{cor:main}, which answers Question~\ref{q2} by proving that the value of human expertise is always equal to the upper bound from \S\ref{sec:simplebounds} if Assumption~\ref{ass:convex} is also satisfied.

\subsection{Simple Bounds} \label{sec:simplebounds}
We begin by establishing an upper bound on $\Delta$ and other simple preliminary  results that hold under Assumption~\ref{ass:main}. These results will be obtained by relating the following three problems:\looseness=-1
\begin{align}
&\max_{\bx \in \mathcal{X}} \min_{\btheta \in \mathcal{U}_\eta} f(\bx,\btheta) \tag{\ref{prob:robust_eta}}\\
    &\max_{\bx \in \mathcal{X}} \min_{\param \in \mathcal{U}} f(\bx,\param) \label{prob:robust}\\
    &\min_{\param \in \mathcal{U}} \max_{\bx \in \mathcal{X}} f(\bx,\param) \label{prob:robust_exchange}
\end{align}
The first problem is a restatement of the reduced robust optimization problem~\eqref{prob:robust_eta} from \S\ref{sec:managerialexpertise}. The second problem~\eqref{prob:robust} is a traditional robust optimization problem. The third problem~\eqref{prob:robust_exchange} is obtained from  \eqref{prob:robust} by interchanging  the maximum and minimum.

 It follows from Assumption~\ref{ass:main} that the optimums in the outer and inner problems of the above three optimization problems are always attained. Moreover, it follows from the standard minimax inequality that the optimal value of the min-max problem~\eqref{prob:robust_exchange} is an upper bound on the optimal value of the max-min problem~\eqref{prob:robust}. Note that Assumptions~\ref{ass:main} and \ref{ass:convex} are insufficient for minimax theorems and do not ensure the existence of saddle points for \eqref{prob:robust};  for example, results like Sion's minimax theorem that ensure the equality of  the optimal values of \eqref{prob:robust} and \eqref{prob:robust_exchange} hold only under additional assumptions such as  convexity of $\mathcal{X}$ and quasi-concavity of $f(\cdot,\param)$.\looseness=-1

 In what follows, we develop an upper bound on $\Delta$ by showing that the optimal value of the min-max problem~\eqref{prob:robust_exchange} is also an upper bound on the optimal value of the reduced robust optimization problem~\eqref{prob:robust_eta}. To establish that upper bound, we first characterize the viable choices for the scalar $\eta$ in the reduced robust optimization problem~\eqref{prob:robust_eta}, that is, the values of the scalar for which the reduced uncertainty set is nonempty. In the following proposition, we formalize the viable choices for $\eta$ by relating the requirement that $\mathcal{U}_\eta \neq \emptyset$ to the optimal value of the min-max problem~\eqref{prob:robust_exchange}.

\begin{proposition} \label{prop:eta}
If Assumption~\ref{ass:main} holds, then $\mathcal{U}_\eta \neq \emptyset$ if and only if   $\eta \ge \min\limits_{\param \in \mathcal{U}} \max\limits_{\bx \in \mathcal{X}} f(\bx,\param)$. 
\end{proposition}
\begin{proof}
Suppose that  $\eta \ge \min_{\param \in \mathcal{U}} \max_{\bx \in \mathcal{X}} f(\bx,\param)$. Then  
\begin{align*}
    \mathcal{U}_\eta &
    \supseteq \left \{ \param \in \mathcal{U}: \max_{\bx \in \mathcal{X}} f(\bx,\param) \le \min_{\hat{\param} \in \mathcal{U}} \max_{\bx \in \mathcal{X}} f(\bx,\hat{\param}) \right \}= \argmin_{\hat{\param} \in \mathcal{U}} \max_{\bx \in \mathcal{X}} f(\bx,\hat{\param})  \neq \emptyset,
\end{align*}
where  the set inclusion follows from the definition of $\mathcal{U}_\eta$ and the supposition on $\eta$, the equality follows from algebra, and the nonemptiness follows from the fact that the optimum of \eqref{prob:robust_exchange} is attained by Assumption~\ref{ass:main}. The converse follows from similar reasoning. 
\end{proof}

 In view of the above, the following Proposition~\ref{prop:bounds} establishes that the optimal value of \eqref{prob:robust_exchange} is an upper bound on the optimal value of the reduced robust optimization problem~\eqref{prob:robust_eta}. It is followed by Corollary~\ref{cor:necessary}, which follows immediately from Proposition~\ref{prop:bounds} and the definition of $\Delta$.
\begin{proposition} \label{prop:bounds}
If Assumption~\ref{ass:main} holds and  $\eta \ge \min\limits_{\param \in \mathcal{U}} \max\limits_{\bx \in \mathcal{X}} f(\bx,\param)$, then $\max \limits_{\bx \in \mathcal{X}} \min \limits_{\param \in \mathcal{U}_\eta} f(\bx,\param)  \le \min \limits_{\param \in \mathcal{U}} \max\limits_{\bx \in \mathcal{X}}  f(\bx,\param)$. 
\end{proposition}
\begin{proof}
It follows from Proposition~\ref{prop:eta} and the fact that $\eta \ge \min_{\param \in \mathcal{U}} \max_{\bx \in \mathcal{X}} f(\bx,\param)$ that $\mathcal{U}_\eta \neq \emptyset$. Therefore, %
\begin{align*}
\max_{\bx \in \mathcal{X}} \min_{\param \in \mathcal{U}_\eta} f(\bx,\param)  \le \min_{\param \in \mathcal{U}_\eta} \max_{\bx \in \mathcal{X}} f(\bx,\param) =  \left[ \begin{aligned}
& \min_{\param \in \mathcal{U}}&&  \max_{\bx \in \mathcal{X}} f(\bx,\param)\\ 
&\textnormal{s.t.}&& \max_{\bx \in \mathcal{X}} f(\bx,\param) \le \eta 
\end{aligned} \right]  = \min_{\param \in \mathcal{U}} \max_{\bx \in \mathcal{X}} f(\bx,\param), 
\end{align*}
where the inequality is the minimax inequality, the first equality follows from the definition of $\mathcal{U}_\eta$, and the second equality follows from the fact that the constraint $\max_{\bx \in \mathcal{X}} f(\bx,\param) \le \eta$ can be removed without affecting the outer minimization problem.
\end{proof}
\begin{corollary} \label{cor:necessary}
If Assumption~\ref{ass:main} holds, then
 $\Delta \le \min \limits_{\param \in \mathcal{U}} \max \limits_{\bx \in \mathcal{X}}  f(\bx,\param) - \max \limits_{\bx \in \mathcal{X}} \min \limits_{\param \in \mathcal{U}} f(\bx,\param).$
\end{corollary}
The above proposition and corollary show that the value of human expertise is linked to the gap between the min-max problem~\eqref{prob:robust_exchange} and the max-min problem~\eqref{prob:robust}. Specifically, Proposition~\ref{prop:bounds} shows that for the optimal value of the reduced robust optimization problem~\eqref{prob:robust_eta} to be greater  than the optimal value of the robust optimization problem~\eqref{prob:robust}, it is \emph{necessary} for there to be a gap between the optimal values of the max-min problem~\eqref{prob:robust} and the min-max problem~\eqref{prob:robust_exchange}. Corollary~\ref{cor:necessary} shows that the value of human expertise  is at most the gap between the optimal values of \eqref{prob:robust_exchange} and \eqref{prob:robust}.\looseness=-1

Corollary~\ref{cor:necessary} is useful because it furnishes us with a simple test for determining whether it is \emph{not} worthwhile to consider nominal curves~\eqref{line:wc_eta_set}. Indeed, Corollary~\ref{cor:necessary} implies that the worst-case performance $\min_{\btheta \in \mathcal{U}_\eta} f(\bx,\btheta)$ for a policy $\bx$ cannot be strictly greater than the optimal value of the robust optimization problem~\eqref{prob:robust}  in settings where the assumptions for minimax theorems are satisfied, such as problems where $\mathcal{X},\mathcal{U}$ are convex sets and $f(\cdot,\cdot)$ is concave in its first argument and convex in its second argument.  Corollary~\ref{cor:necessary} also shows that the worst-case performance $\min_{\btheta \in \mathcal{U}_\eta} f(\bx,\btheta)$ for a policy $\bx$ cannot be much greater than the optimal value of \eqref{prob:robust} if the gap between the optimal values of \eqref{prob:robust} and \eqref{prob:robust_exchange} is small.  On the positive side, Corollary~\ref{cor:necessary} leaves open the possibility that the value of human expertise can be large in settings where pure strategies for the robust optimization problem~\eqref{prob:robust} are highly suboptimal, i.e., settings where the gap between the optimal values of \eqref{prob:robust} and \eqref{prob:robust_exchange} is large.

\subsection{Main Result} \label{sec:maintheorem}
Corollary~\ref{cor:necessary} provides a simple upper bound on the value of human expertise which holds whenever  Assumption~\ref{ass:main} is satisfied. In particular, Corollary~\ref{cor:necessary} shows that the  value of human expertise has the potential to be strictly positive---equivalently, the optimal value of the reduced robust optimization problem~\eqref{prob:robust_eta} can be larger than the optimal value of the robust optimization problem~\eqref{prob:robust}---when the gap between the optimal values of the max-min problem~\eqref{prob:robust} and the min-max problem~\eqref{prob:robust_exchange} is large. But is that upper bound on $\Delta$ always attainable, or is the best-case value of human expertise less in practice?\looseness=-1

In this subsection, we answer that question by showing that the upper bound from Corollary~\ref{cor:necessary}  is {tight} and always attained under Assumption~\ref{ass:convex}. 
Our main result  is a consequence of the following theorem, which states a property about optimal solutions for max-min problems. 
\begin{theorem} \label{thm:fundamental}
If Assumptions~\ref{ass:main} and \ref{ass:convex} hold and $\eta = \min \limits_{\param \in \mathcal{U}} \max \limits_{\bx \in \mathcal{X}} f(\bx,\param)$, then 
$\max \limits_{\bx \in \mathcal{X}} \min \limits_{\param \in \mathcal{U}_\eta} f(\bx,\param) = \min \limits_{\param \in \mathcal{U}} \max \limits_{\bx \in \mathcal{X}} f(\bx,\param).$
\end{theorem}
 The above theorem establishes that a policy that performs best against all of the optimal solutions of the min-max problem~\eqref{prob:robust_exchange} yields an optimal value that is equal to the optimal value of the min-max problem~\eqref{prob:robust_exchange}.
 The value of  Theorem~\ref{thm:fundamental} lies in its generality: it does not require convexity in $\mathcal{X}$ or quasi-concavity of $f(\cdot,\param)$, does not assume that $\mathcal{U}^*$ is a singleton, and does not require that the sets $\mathcal{X}$ and $\mathcal{U}$ are finite-dimensional.   Theorem~\ref{thm:fundamental} thus extends to a wide range of applications, such as settings where the parameter is a probability measure or the policies are infinite-dimensional in the context of dynamic optimization. The following proof of Theorem~\ref{thm:fundamental} is simple and based on elementary facts about topological spaces and convexity.

\begin{proof}[Proof of Theorem~\ref{thm:fundamental}.]
Let $\eta = \min_{\param \in \mathcal{U}} \max_{\bx \in \mathcal{X}} f(\bx,\param)$ and $\mathcal{U}^* \coloneqq \argmin_{\param \in \mathcal{U}} \max_{\bx \in \mathcal{X}} f(\bx,\param)$. It  follows from Assumption~\ref{ass:convex} that $\mathcal{U}^*$ is convex, and it follows from algebra that $\mathcal{U}^* = \{ \param \in \mathcal{U}: \max_{\bx \in \mathcal{X}} f(\bx,\param) \le \eta \} = \mathcal{U}_\eta$.  We have two cases to consider:
\begin{itemize}
\item  \underline{Case 1}: Suppose that there exists an $\hat{\bx} \in \mathcal{X}$ that satisfies $f(\hat{\bx},\param) = \eta$ for all $\param \in \mathcal{U}^*$. Then 
\begin{align*}
\min_{\param \in \mathcal{U}} \max_{\bx \in \mathcal{X}} f(\bx,\param)= \min_{\param \in \mathcal{U}^*} \max_{\bx \in \mathcal{X}}f(\bx,\param) =  \min_{\param \in \mathcal{U}^*} f(\hat{\bx},\param) \le \max_{\bx \in \mathcal{X}} \min_{\param \in \mathcal{U}^*}  f(\bx,\param).
\end{align*}
Combining the above inequality with the fact that $\mathcal{U}^* = \mathcal{U}_\eta$ and Proposition~\ref{prop:bounds}, we have 
\begin{align*}
    \max \limits_{\bx \in \mathcal{X}} \min \limits_{\param \in \mathcal{U}^*} f(\bx,\param) = \min \limits_{\param \in \mathcal{U}} \max \limits_{\bx \in \mathcal{X}}  f(\bx,\param).
\end{align*}
\item \underline{Case 2}: Suppose that there does not exist an $\hat{\bx} \in \mathcal{X}$ that satisfies  $f(\hat{\bx},\param) = \eta$ for all $\param \in \mathcal{U}^*$. Then it follows from Assumption~\ref{ass:main} and from a direct application of basic results about topological spaces (see Appendix~\ref{app:topology_application}) that there must exist a finite subset $\{\param_1,\ldots,\param_K \} \subseteq \mathcal{U}_\eta$  that satisfies\looseness=-1
\begin{align}
    \bigcap_{k=1}^K \left \{ \bx \in \mathcal{X}:f(\bx,\param_k) = \eta \right \} = \emptyset.   \label{line:emptyset}
\end{align}
Since $\mathcal{U}_\eta$ is convex, it must be the case that $\hat{\param} \coloneqq \frac{1}{K} \sum_{k=1}^K \param_k$ is an element of $\mathcal{U}^*$.  Moreover, since $\hat{\param} \in \mathcal{U}^*$, it must be the case that there exists $\hat{\bx} \in \mathcal{X}$ that satisfies $f(\hat{\bx},\hat{\param}) = \eta$. However, we observe that
\begin{align*}
   \eta = f(\hat{\bx}, \hat{\param}) =  f\left(\hat{\bx}, \frac{1}{K} \sum_{k=1}^K \param_k \right) \le   \frac{1}{K} \sum_{k=1}^K f(\hat{\bx}, \param_k)  < \frac{1}{K} \sum_{k=1}^K \eta = \eta,
\end{align*}
where the first inequality follows from Assumption~\ref{ass:convex}, and 
the strict inequality follows from  line~\eqref{line:emptyset} (which implies that there exists $k \in \{1,\ldots,K\}$ such that $f(\hat{\bx},\param_k) < \eta$) and from the fact that $\param_1,\ldots,\param_K \in \mathcal{U}^*$ (which implies for each $k \in \{1,\ldots,K\}$ and for all $\bx \in \mathcal{X}$ that $f(\bx,\param_k) \le \eta$). We thus have a contradiction, which implies we cannot be in Case 2. 

\end{itemize}
Because those cases are exhaustive, our proof of Theorem~\ref{thm:fundamental} is complete. \end{proof}
Combining the above theorem with Proposition~\ref{prop:eta} and Corollary~\ref{cor:necessary}, we obtain the main result of this paper:\looseness=-1
\begin{corollary}\label{cor:main}
If Assumptions~\ref{ass:main} and \ref{ass:convex} hold, then $\Delta =  \min \limits_{\param \in \mathcal{U}} \max \limits_{\bx \in \mathcal{X}}  f(\bx,\param) - \max \limits_{\bx \in \mathcal{X}} \min \limits_{\param \in \mathcal{U}} f(\bx,\param)$.
\end{corollary}
\noindent The above corollary shows under Assumptions~\ref{ass:main} and \ref{ass:convex} that the upper bound from Proposition~\ref{prop:bounds} is always attained. That is, it shows that if $\eta$ is the smallest possible scalar for which the reduced uncertainty set is nonempty (Proposition~\ref{prop:eta}), then the gap between the optimal values of the reduced robust optimization problem~\eqref{prob:robust_eta} and the robust optimization problem~\eqref{prob:robust} is equal to the gap between the optimal values of  \eqref{prob:robust_exchange} and \eqref{prob:robust}.

Corollary~\ref{cor:main} has a number of practical implications. It shows affirmatively that if there is a gap between the optimal values of \eqref{prob:robust} and \eqref{prob:robust_exchange}, and if Assumptions~\ref{ass:main} and \ref{ass:convex} are satisfied, then there exist policies with worst-case  performance guarantees that are strictly greater than the optimal value of \eqref{prob:robust} if $\bar{\btheta} \in \mathcal{U}$ and the optimal value of \eqref{prob:true} is small. In other words, Corollary~\ref{cor:main} provides an exact characterization of the problems for which nominal curves~\eqref{line:wc_eta_set} can be practically informative. 
More generally, Theorem~\ref{thm:fundamental} shows that incorporating an upper bound on the optimal value of the nominal problem~\eqref{prob:true} into the uncertainty set of a robust optimization problem~\eqref{prob:robust} can induce an interpolation between optimal pure strategies and mixed strategies. Specifically,  Theorem~\ref{thm:fundamental} shows that incorporating an upper bound on \eqref{prob:true} modifies an uncertainty set in such a way that can make the optimal pure strategy for the reduced robust optimization problem~\eqref{prob:robust_eta} equal in optimal value to the optimal mixed strategy for the original robust optimization problem \eqref{prob:robust}. This interpretation is demonstrated through the following toy example.\looseness=-1

\begin{example} \label{example:mainresult}
    Let $\mathcal{X}= \{-1, 1\}$, $\mathcal{U} = [-1,1]$, and $f(x,\theta) =\theta x$. Then the optimal value over pure strategies and the optimal pure strategies are
    \begin{align*}
        &\max_{x \in \mathcal{X}} \min_{\theta \in \mathcal{U}} \theta x = -1 \text{ and }  \argmax_{x \in \mathcal{X}} \min_{\theta \in \mathcal{U}} \theta x = \{-1,1\}.
    \end{align*}
Letting $\mathcal{P}(\mathcal{X})$ denote the set of mixed strategies for the outer problem, and letting $\delta_x$ denote the Dirac delta measure at $x$, we observe that the optimal value over mixed strategies and the optimal mixed strategy  are
    \begin{align*}
       &\max_{\mu \in \mathcal{P}(\mathcal{X})} \min_{\theta \in \mathcal{U}} \Exp_{X \sim \mu}[\theta X] =     \min_{\theta\in \mathcal{U}} \max_{x \in \mathcal{X}}  \theta x = 0 \text{ and } \argmax_{\mu \in \mathcal{P}(\mathcal{X})} \min_{\theta \in \mathcal{U}} \Exp_{X \sim \mu} [\theta X] = \{\sfrac{1}{2} \delta_{-1} + \sfrac{1}{2} \delta_{1}\}.
    \end{align*}
    It follows from Proposition~\ref{prop:eta} that $\mathcal{U}_\eta \neq \emptyset$ if and only if $\eta \in [\min_{\theta \in \mathcal{U}} \max_{x \in \mathcal{X}} \theta x, \infty) = [0,\infty)$, and for each such $\eta$ the reduced uncertainty set is 
    \begin{align*}
   \mathcal{U}_\eta = \left\{ \theta \in [-1,1]: \max_{x \in \{-1,1\}} \theta x \le \eta \right\} = \begin{cases}
   [-1,1], &\text{if } \eta \ge 1, \\ 
   [-\eta,\eta], &\text{if }  0 \le \eta < 1.
   \end{cases} 
    \end{align*}
   Therefore, for all $\eta \in [0,\infty)$, the optimal value over pure strategies and the optimal pure strategies for the reduced robust optimization problem are\looseness=-1
    \begin{align*}
     \max_{x \in \mathcal{X}} \min_{\theta  \in \mathcal{U}_\eta} \theta x      &= 
     \begin{cases}
     -1, &\text{if } \eta \ge 1, \\ 
     -\eta, &\text{if } 0 \le \eta < 1 
     \end{cases}   
\text{ and } \argmax_{x \in \mathcal{X}} \min_{\theta  \in \mathcal{U}_\eta} \theta x = \{-1,1\}.
    \end{align*}
    Thus, in the extreme case where $\eta = 0$, the optimal pure strategy for the reduced robust optimization problem is equal in optimal value to the optimal mixed strategy for the original robust optimization problem. \qed
\end{example}

In Example~\ref{example:mainresult}, the set of optimal solutions for the reduced robust optimization problem is the same for all $\eta$, in the sense that $\argmax_{x \in \mathcal{X}} \min_{\theta  \in \mathcal{U}_\eta} \theta x = \{-1,1\}$ for all $\eta \in [0,\infty)$. However, it is possible for the set of optimal solutions for \eqref{prob:robust_eta} to change with $\eta$; see Figure~\ref{fig:assortment_small} in \S\ref{sec:assortment:results}. The following Example~\ref{example:nonconvex} shows that  Corollary~\ref{cor:main} can fail to hold if Assumption~\ref{ass:convex} is violated. %
\begin{example} \label{example:nonconvex}
    Let $\mathcal{X} = \mathcal{U} = \{-1, 1\}$ and $f(x,\theta) = \theta x$. Then 
    \begin{gather*}
        \max_{x \in \mathcal{X}} \min_{\theta  \in \mathcal{U}} \theta x = -1 \text{ and }     \min_{\theta  \in \mathcal{U}} \max_{x \in \mathcal{X}} \theta x = 1. 
    \end{gather*}
Assumption~\ref{ass:convex} is violated because $\mathcal{U}$ is non-convex, and we observe for all $\eta \ge   \min_{\theta  \in \mathcal{U}} \max_{x \in \mathcal{X}} \theta x = 1$ that 
    \begin{align*}
     \max_{x \in \mathcal{X}} \min_{\theta  \in \mathcal{U}_\eta} \theta x =      \max_{x \in \{-1,1\}} \min_{\theta  \in \{-1,1\}: \max_{\hat{x} \in \{-1,1\}} \theta  \hat{x} \le \eta } \theta x = \max_{x \in \{-1,1\}} \min_{\theta  \in \{-1,1\}} \theta  x = -1,
    \end{align*}
    which implies that the equality from Corollary~\ref{cor:main} is not satisfied.
    \qed
\end{example}

In summary,  Corollary~\ref{cor:main} gives answers to the questions of \emph{if} and \emph{by how much} a belief that the optimal value of the nominal problem~\eqref{prob:true} is unlikely to be large can lead to performance guarantees that are less conservative than those obtained by robust optimization. In the following sections, we apply Corollary~\ref{cor:main} to practical settings to analyze the value of human expertise and shed light on the usefulness of nominal curves.\looseness=-1


\section{Identifiability in Data-Driven Assortment Optimization} \label{sec:assortment:example1}

In \S\ref{sec:assortment:example1}, we consider a parent company that seeks to identify a new assortment to recommend to a local store using  the transactional sales data generated by the store's past assortments. We present an example in which it is impossible for the parent company to identify an assortment that is guaranteed to outperform the store's best past assortment across all of the random utility maximization models that are consistent with the data generated by the store's past assortments.  We then show for the example that nominal curves~\eqref{line:wc_eta_set}  make it possible to identify assortments that in the worst case do not perform much worse than the store’s best past assortment, but are guaranteed to  strictly outperform the store’s best past assortment if the expected revenue of the store’s best past assortment happens to be close to optimal. %

\subsection{Problem Setup} \label{sec:assortment:setting}
 Assortment optimization is a class of problems from revenue management in which the goal is to select a subset of products for a store to offer to its customers that maximizes expected revenue.  Let $1,\ldots,n$ denote the products that the store  may offer, and the revenues of the products are given by $r_1,\ldots,r_n > 0$.  Let the no-purchase option be denoted by index $0$ with $r_0 = 0$,  and let  $\mathscr{S} \equiv \{S \subseteq \{0,\ldots,n\}: 0 \in S \}$ denote the set of all assortments.   Given $S \in \mathscr{S}$ and $i \in S$,  let $\bar{\mathbb{P}}(i | S) \in [0,1]$ denote the proportion of the store's customers that purchase product $i$ when offered assortment $S$. If the store's true discrete choice model $\bar{\Prb}$ were known, then an assortment that maximizes the store's expected revenue would be obtained by solving
    \begin{equation} \label{prob:assortment_general}
        \begin{aligned}
            &\underset{S \in \mathscr{S}}{\text{maximize}} && \sum_{i \in S} r_i \bar{\Prb}(i \mid S)
        \end{aligned}
    \end{equation}

We take the role of a parent company that wishes to find a new assortment to recommend to the local brick-and-mortar store to offer to its customers. 
 The  true discrete choice model $\bar{\Prb}$ that captures the local customer preferences at the store is unknown, and the parent company only has historical sales data from the past assortments that were chosen by the store. %
 Moreover, the store manager is risk-averse and reluctant to experiment with new assortments that might lead to a decline in expected revenue.
The goal of the parent company is to identify a new assortment that can be trusted to outperform the store's best past assortment     across all of the stochastically rational (that is, random utility maximization) discrete choice models that are consistent with the historical sales data generated by the store's past assortments.    

 The local store in our example has four products that it may offer,  denoted by $1,2,3,4$, and the per-unit revenues of the products are  $$r_1 = \$2,\; r_2 = \$10,\; r_3 = \$31,\;r_4 = \$40.$$
  We assume that the true discrete choice model $\bar{\Prb}$  of the local store is unknown, but the parent company has information about the true discrete choice model from aggregated transactional sales data generated from the past assortments offered by the store to its customers.   Specifically, the store has previously offered two different assortments, $$S_1 = \{0,2,3\},\; S_2 = \{0,1,3,4\},$$ and from these assortments the parent company observed from the transactional sales data that\looseness=-1
    \begin{align*}
    \bar{\Prb}(0 | S_1)  = 0&, \quad 
  \bar{\Prb}(2 | S_1) =     \bar{\Prb}(3 | S_1)  = \sfrac{1}{2},\\
        \bar{\Prb}(0 | S_2) =     \bar{\Prb}(3 | S_2)  = 0&, \quad  \bar{\Prb}(1 | S_2) =     \bar{\Prb}(4 | S_2)  = \sfrac{1}{2}.
    \end{align*}

    We assume that the two past assortments were offered by the store manager for a sufficiently long time such that there is little statistical uncertainty in these estimates.  As is typical in the revenue management literature, we also assume that  $\bar{\Prb}$ is a random utility maximization (RUM) model, which is a general class of discrete choice models that  is consistent with stochastic rationality \citep{blockmarschak} and subsumes most parametric families of discrete choice models studied in revenue management.  Define the uncertainty set of all RUM models that are consistent with the transactional sales data generated by the past assortments 
    \begin{align}
        \mathcal{U} \triangleq  \left \{  \Prb \in \mathscr{P}:  \quad \begin{aligned}
    {\Prb}(0 | S_1)  = 0&, \quad 
  {\Prb}(2 | S_1) =     {\Prb}(3 | S_1)  = \sfrac{1}{2}\\
        {\Prb}(0 | S_2) =     {\Prb}(3 | S_2)  = 0&, \quad  {\Prb}(1 | S_2) =     {\Prb}(4 | S_2)  = \sfrac{1}{2}
        \end{aligned}\right \}, \label{line:assortment_uncertaintyset}
    \end{align}
    where $\mathscr{P}$ denotes the set of all RUM models over the universe of products. The uncertainty set of the form \eqref{line:assortment_uncertaintyset}, which was introduced by \citet{farias2013nonparametric}, can be viewed as very general, as it 
is  guaranteed to contain the true discrete choice model $\bar{\Prb}$ if there is no noise in the transactional sales data and if $\bar{\Prb}$ is a RUM model.\footnote{Both of these assumptions can be relaxed by modifying the set from line~\eqref{line:assortment_uncertaintyset} to deviate from the historical choice probabilities by at most a given radius; see, for example, \citet[\S\S3.3 and 5.1.3]{farias2013nonparametric}.}\looseness=-1

 In view of the above, the identification problem faced by the parent company can be formally stated as follows. First,  we observe  that the expected revenues of the store's past assortments are
      \begin{align*}
        \sum_{i \in S_1} r_i \bar{\Prb}(i | S_1) = r_2 \frac{1}{2} + r_3 \frac{1}{2} = \$20.5, \quad   \sum_{i \in S_2} r_i \bar{\Prb}(i | S_2) = r_1 \frac{1}{2} + r_4 \frac{1}{2} = \$21.
    \end{align*} 
    The problem of identifying a new assortment that is guaranteed to outperform the store's best past assortment     across all of the RUM models that are consistent with the transactional sales data generated by the store's past assortments thus can be cast as the problem of  finding an assortment $S \in \mathscr{S}$ that satisfies 
      \begin{align}
 \sum_{i \in S} r_i {\Prb}(i | S) > \$21 \quad \forall \Prb \in \mathcal{U} \label{line:identification}
   \end{align}
where $\$21 = \max \{\$20.5,\$21\}$ is the expected revenue of the store's best past assortment.

\subsection{The Identification Problem Has No Solution} \label{sec:assortment:identification_impossible}

There are examples in which there exist assortments that outperform the store's best past assortment across all RUM models that are consistent with transactional sales data generated by the past assortments~\cite[\S3]{sturt2025value}.   However, the example from \S\ref{sec:assortment:setting} is not one of them: we show in \S\ref{sec:assortment:results} and Appendix~\ref{appx:assortment_extra} for each assortment $S \in \mathscr{S}$ that\looseness=-1
   \begin{align*}
 \min_{\Prb \in \mathcal{U}} \sum_{i \in S} r_i {\Prb}(i | S) \le \$21.
   \end{align*}
   As such, there is no solution to the identification problem in \S\ref{sec:assortment:setting}, in the sense that there does not exist an assortment that satisfies \eqref{line:identification}. This implies that there is no new assortment that the parent company can suggest to the store that can be trusted in the worst case to increase the store's expected revenue. 
   
   What should the parent company do in situations like that described above? 
The most common approach in the literature is to impose parametric assumptions on the structure of RUM models. Over the past two decades, an extensive literature in revenue management has studied parametric subclasses of RUM models such as the multinomial logit and nested logit models. By replacing the uncertainty set of all RUM models that fit the historical sales data $\mathcal{U}$ with only  RUM models from a specific parametric family (see \S\ref{sec:lit}), it may be possible to identify an assortment that satisfies \eqref{line:identification} over the restricted uncertainty set. %
   
   We take a different approach: rather than making parametric assumptions on the structure of the true discrete choice model, we instead use a belief about the process that generated the past assortments that the store offered to its customers. Indeed, recall that in the setting described above, we assumed that the two past assortments $S_1$ and $S_2$ were offered by the local store long enough to have accurate estimates of the choice probabilities. This suggests that the  two assortments may not have been chosen randomly or adversarially; rather, these assortments were  chosen by store managers informed by their  private information about local preferences from observing and interacting with typical store customers \citep{kok2008assortment,farias2017building}. As such,  the parent company may believe it is reasonable to assume that these past assortments are not highly suboptimal.  
   Motivated by this, we will use nominal curves to search for assortments that perform nearly as well as $S_2$ in the worst case, but are guaranteed to outperform $S_2$ if the expected revenue from $S_2$ happens to be close to the optimal value of the assortment optimization problem~\eqref{prob:assortment_general} with the true but unknown  discrete choice model $\bar{\Prb}$.\looseness=-1

       Our approach using nominal curves described above, similarly to the more common approach of imposing parametric assumptions on the set of RUM models, involves making assumptions about the true assortment optimization problem~\eqref{prob:assortment_general}.    However, our approach of using nominal curves is potentially attractive in practice, since it may be easier to justify the assumption that the local store used private information to select their past assortments than  justifying a parametric assumption on the structure of RUM models, e.g., that $\bar{\Prb}$ follows a multinomial logit model. Moreover, our approach provides explicit guarantees that hold if the belief about the near-optimality of the best past assortment $S_2$ is incorrect, whereas the misspecification error from a parametric RUM model in assortment optimization is difficult to analyze in general. 
       
       \subsection{Numerical Results} \label{sec:assortment:results}

 In Figure~\ref{fig:assortment_small}, we present the nominal curves corresponding to a subset of feasible assortments in the example from \S\ref{sec:assortment:setting}. Specifically, the figure plots the nominal curves $v_S(\eta) \triangleq \min_{\Prb \in \mathcal{U}_\eta} \sum_{i \in S} r_i \Prb(i \mid S)$ for the two past assortments $\{0,2,3\}$ and $\{0,1,3,4\}$ as well as a new assortment $ \{0,2,3,4\}$, where the reduced uncertainty set\looseness=-1%
     \begin{align*}
    \mathcal{U}_\eta =  \left \{  \Prb \in \mathcal{U}:  \; 
            \max_{S \in \mathscr{S}} \sum_{i \in S} r_i {\Prb}(i \mid S) \le \eta\right \} %
    \end{align*}
  is the set of all  the discrete choice models from the uncertainty set~\eqref{line:assortment_uncertaintyset} that, if true, would lead to an assortment optimization problem~\eqref{prob:assortment_general} with an optimal expected revenue of at most $\eta$. 
It is shown in Appendix~\ref{appx:assortment_extra} that the reduced robust optimization problem
\begin{align*}
            \max_{S \in \mathscr{S}} \min_{\Prb \in \mathcal{U}_\eta} \sum_{i \in S} r_i {\Prb}(i \mid S) %
\end{align*}
has $S_2$ as an optimal solution when $\eta\ge \tilde{\eta}$ %
and has $\{0,2,3,4\}$  as an optimal solution when $\eta\le \tilde{\eta}$ and $\mathcal{U}_\eta \neq \emptyset$, where $\tilde{\eta} \approx \$33.778$. %
The quantities  $v_S(\eta)$ are calculated by solving the linear program~(2) from \citet[\S2.4]{farias2013nonparametric} with the additional constraints that $\sum_{i \in S'} r_i \Prb(i \mid S') \le \eta$ for all $S' \in \mathscr{S}$. %

          \begin{figure}[t]
        \centering
        \includegraphics[width=0.9\textwidth]{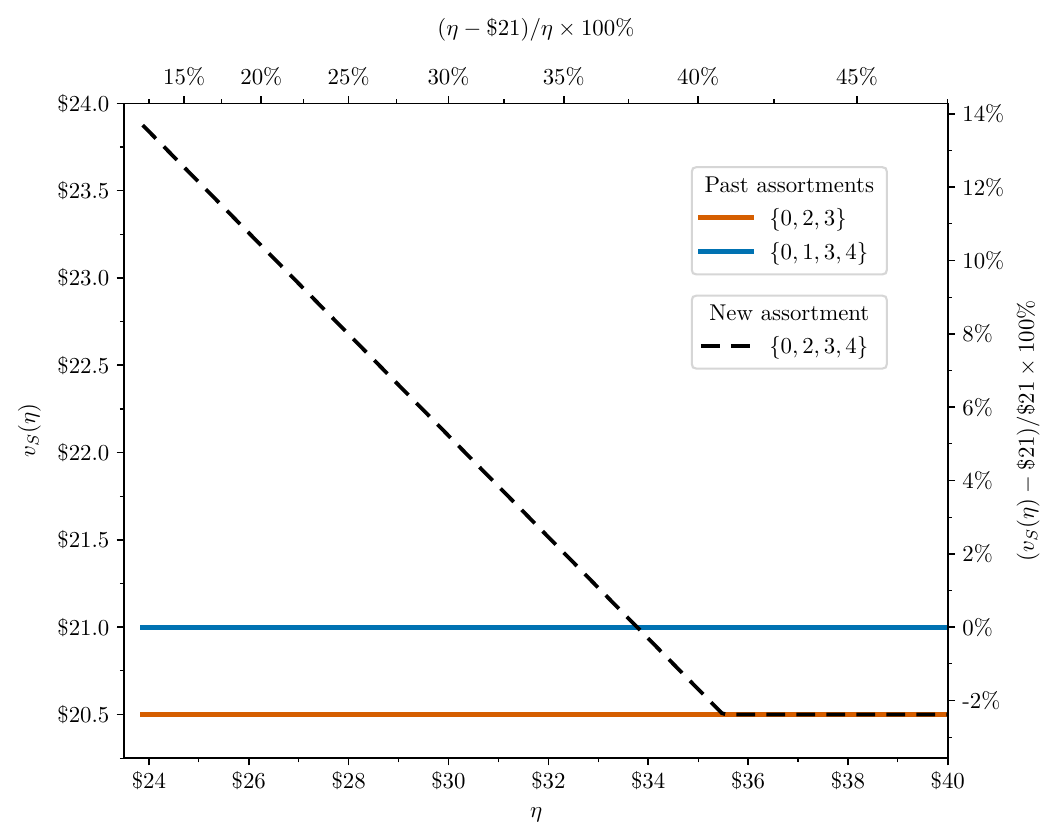}
        \caption{Nominal curves from numerical example in \S\ref{sec:assortment:results}.\looseness=-1}
         \label{fig:assortment_small}
    \end{figure}
    
To make sense of Figure~\ref{fig:assortment_small}, we begin by interpreting the labels. The values of $\eta$ at the bottom of the plot are possible upper bounds on the optimal value of the nominal problem~\eqref{prob:assortment_general}, and the values $(\eta -\$ 21)/\eta \times 100\%$ at the top of the plot show possible upper bounds on the optimality gap of the store's best past assortment.   The values of $v_S(\eta)$ on the left of the plot are the worst-case expected revenue of each assortment $S$ if the optimal value of the nominal problem~\eqref{prob:assortment_general} happens to be less than or equal to $\eta$. The values of $(v_S(\eta) - \$21)/\$21 \times 100\%$ on the right of the plot show the worst-case percentage increase in expected revenue from switching the store's best past assortment $S_2 = \{0,1,3,4\}$ to assortment $S$ if the optimal value of the nominal problem~\eqref{prob:assortment_general} happens to be less than or equal to $\eta$. 

For example, consider the line from Figure~\ref{fig:assortment_small} corresponding to the new assortment $\{0,2,3,4\}$ at the point where $\eta = \$30$. At that point, we observe that $(\eta - \$21) / \eta \times 100\% =  30\%$, and it is shown in Figure~\ref{fig:assortment_small} that $v_{\{0,2,3,4\}}(\$30) \approx \$22.097$. We can interpret this point of the nominal curve as saying that if the expected revenue of the store's best past assortment happened to be within 30\% of optimal, then the assortment $\{0,2,3,4\}$ is guaranteed to generate an expected revenue of at least $\$22.097$, or equivalently,  guaranteed to increase the store's expected revenue by at least $(v_{\{0,2,3,4\}}(\$30) - \$21)/\$21 \times 100\% = 5.223\%$. By similar reasoning, Figure~\ref{fig:assortment_small} shows that $\{0,2,3,4\}$ is guaranteed to increase the store's expected revenue by at least $7.988\%$ if the store's best past assortment happened to be within 25\% of optimal, and guaranteed to increase the store's expected revenue by at least 10.407\% if the store's best past assortment happened to be within 20\% of optimal.\looseness=-1

From the perspective of the parent company,  the main takeaways from Figure~\ref{fig:assortment_small} are the following. Recall that the expected revenue of the best past assortment $S_2$ is $\$21$, which must be within $(\$40 - \$21)/\$40 \times 100\% = 47.5\%$ of optimal since $r_4 = \$40$ is the revenue of the most expensive product. We observe from  the nominal curve for the assortment $\{0,2,3,4\}$ that if the optimal value of the assortment optimization problem~\eqref{prob:assortment_general} happens to be less than approximately $\$33.778$---equivalently, if the best past assortment happens to be within $(\$33.778 - \$21) / \$33.778 \times 100\% \approx 37.8\%$ of optimal---then the assortment $\{0,2,3,4\}$ is guaranteed to generate an expected revenue that is higher than the expected revenue of the store's best past assortment $S_2$. If the parent company's belief that the best past assortment is not highly suboptimal is correct, then the new assortment $\{0,2,3,4\}$ has a worst-case expected revenue that can be as large as $(\$23.875 - \$21) / \$21 \times 100\% = 13.7\%$ higher than the expected revenue of the best past assortment; this occurs at $\eta = \$23.875$. If the parent company's belief is incorrect, then the decrease in expected revenue from switching from the best past assortment to $\{0,2,3,4\}$ is at most $- ( \$20.5 - \$21) / \$21 \times 100\% = 2.38\%$.\looseness=-1

\emph{Should the parent company recommend the store to switch to the new assortment $\{0,2,3,4\}$?} The answer to this question is ultimately a  judgment call for the parent company. That said,  we find it reasonable to expect in practice that the parent company would recommend the assortment $\{0,2,3,4\}$  to the store. Indeed,  the nominal curve for the new assortment shows that the worst-case decrease in expected revenue for the store is at most  2.38\%, and that decrease can occur only if the store's best past assortment was highly suboptimal, which the parent company was assumed to believe is unlikely. Given that the downside risk is mild and believed to be unlikely, and given that the nominal curve shows that the upside for switching to $\{0,2,3,4\}$ may be considerable if the store's best past assortment was not highly suboptimal, the parent company and the store may view the new assortment as attractive. %

We conclude with a few additional remarks. First,  we observe from Figure~\ref{fig:assortment_small} that the nominal curves corresponding to the two past assortments $S_1$ and $S_2$  remain constant regardless of $\eta$. This follows from the construction of the uncertainty set~\eqref{line:assortment_uncertaintyset}, which implies that $v_{S_1}(\eta) =  \sum_{i \in S_1} r_i \bar{\Prb}(i | S_1) $ and $v_{S_2}(\eta)=  \sum_{i \in S_2} r_i \bar{\Prb}(i | S_2) $ for all $\eta$ such that $\mathcal{U}_\eta \neq \emptyset$. Second, the range of values for $\eta$ shown in Figure~\ref{fig:assortment_small} begins at $\eta = \$23.875$  because $\mathcal{U}_\eta$ was found to be empty for $\eta < \$23.875$. Figure~\ref{fig:assortment_small} also shows that $v_{\{0,2,3,4\}}(\$23.875) = \$23.875$. These results are consistent with  Proposition~\ref{prop:eta} and Corollary~\ref{cor:main} and imply that $\min_{\Prb \in \mathcal{U}} \max_{S \in \mathscr{S}} \sum_{i \in S} r_i \Prb(i | S) = \$23.875$. Third, the $x$-axis ends at $\eta = \$40$ because the optimal value of the assortment optimization problem~\eqref{prob:assortment_general} is at most the revenue of the most expensive product, $r_4 = \$40$.\looseness=-1

Nominal curves for assortments not shown in Figure~\ref{fig:assortment_small} can be found in Appendix~\ref{appx:assortment_extra}.

\section{Combinatorial Optimization with Budget Uncertainty Sets} \label{sec:applications:optimizercurse}
In \S\ref{sec:assortment:example1}, we considered an application  where the decision maker's belief about the optimal value of the nominal problem derived from a belief that it is unlikely that the  best past policy was highly suboptimal. In this section, we consider a different application in which the  decision maker's belief may derive from domain expertise and intuition. Specifically, we consider combinatorial optimization problems with uncertain  cost coefficients, where the decision maker believes that it  is unlikely that the optimal value of the true problem is small.\footnote{\S\ref{sec:applications:optimizercurse} focuses on nominal problems that are  minimization problems, and so the belief that the optimal value of the nominal problem is unlikely to be large is reversed to a belief that the optimal value is unlikely to be small. The results from \S\S\ref{sec:managerialexpertise}-\ref{sec:mainresults}  apply to minimization problems by negating the objective function.  } We begin in \S\ref{sec:applications:optimizercurse:setting} by formalizing the application setting. In \S\ref{sec:structure:budget}, we establish theoretical conditions for this application under which relatively loose beliefs about the optimal value of the nominal problem lead to improved performance guarantees compared to those that can be obtained from robust optimization.  In \S\ref{sec:motivatingexample_shortest_path}, we showcase the value of human expertise and the structural results through numerical experiments on shortest path problems.\looseness=-1

\subsection{Problem Setting} \label{sec:applications:optimizercurse:setting}
We consider combinatorial optimization programs of the form 
\begin{equation} \label{prob:lp}
\begin{aligned}
&\min_{I \in \mathcal{I}} \sum_{j \in I}\bar{c}_{j}
\end{aligned}
\end{equation}
where  $\bar{\bc} \ge \bzero$ and $\mathcal{I} \subset 2^{\{1,\ldots,n\}}$ is a collection of subsets of $\{1,\ldots,n\}$ that satisfies $\emptyset \notin \mathcal{I}$.  
For example, in the context of network problems such as bipartite matching, shortest path, and traveling salesman problems, each $I \in \mathcal{I}$ refers to the edges that are activated, and the cost coefficient $\bar{c}_{j} \ge 0$ denotes the cost of activating edge $j$.\looseness=-1

We focus on combinatorial optimization problems where the true cost coefficients $\bar{\bc}$  are unknown and represented by an uncertainty set $\mathcal{U}$. 
The study of uncertainty sets in the context of combinatorial optimization with unknown cost coefficients  has a rich history dating back to \citet{kouvelis1996robust}. Numerous approaches to designing uncertainty sets have been proposed  for the context of combinatorial optimization motivated by domain knowledge, risk aversion, and probabilistic guarantees; see \citet{aolaritei2026diffusion}.  %
One of the most widely studied choices of the uncertainty set in combinatorial optimization with uncertain cost coefficients is the so-called \emph{budget uncertainty set} \citep[\S3]{MelvynSim2003}. It is defined as\looseness=-1 
\begin{align}
    \mathcal{U} &= \left \{ \bc \in \R^n: \;
        \text{there exists } \bz \in [0,1]^n \text{ such that } \sum_{j=1}^n z_j \le \Gamma \text{ and }c_j = \hat{c}_j + d_{j} z_j \text{ for all } j \right \} \notag \\
        &= \left \{ \hat{\bc}+ \bd \odot \bz: \; \bz \in [0,1]^n \text{ and } \sum_{j=1}^n z_j \le \Gamma \right \}\label{line:budget_uncertaintyset} 
\end{align}
where $\hat{\bc} \ge \bzero$ is a lower bound estimate of the true  cost coefficients, $\bd > \bzero$ is an upper bound on the deviations of the cost coefficients, and the integer $\Gamma \ge 1$ controls the  number of cost coefficients that can achieve their maximum deviation. As an example, consider a shortest path or traveling salesman problem corresponding to emergency vehicle routing during a flash flood; in that example,  $\hat{c}_j$ equals the travel time through two locations denoted by edge $j$ during normal weather conditions, and $\bar{c}_j$ equals the increased travel time.\looseness=-1

In this section,  we consider applications where the decision maker believes that the optimal value of \eqref{prob:lp}---that is, the total cost of the optimal solution that we would have chosen if the true cost coefficients were known---is unlikely to be small. %
In the earlier example of emergency routing in the flash flood, for instance, the exact locations of the flood may be uncertain. Nonetheless, the dispatcher may  believe, based on years of experience dispatching during storms and knowledge of the local terrain, that the flooding is likely to be sufficiently   widespread  that the minimum total travel times will be increased compared to what can normally be achieved in the absence of a storm. A precise statement of such beliefs will be made in \S\ref{sec:structure:budget}.  \looseness=-1

For the problem setting described above, the reduced robust optimization problem takes the form
\begin{equation*}
\min_{I \in \mathcal{I}} \max_{\bc \in \mathcal{U}_\eta} \sum_{j \in I}  c_j \text{ where } \mathcal{U}_\eta = \left \{\bc \in \mathcal{U}: \min_{I \in \mathcal{I}}\sum_{j \in I} c_j \ge \eta  \right \}.
\end{equation*}
In the case of the budget uncertainty set, the above reduced robust optimization problem can be rewritten as\looseness=-1
\begin{equation} \label{prob:lp_eta}
\min_{I \in \mathcal{I}} \max_{\bz \in \mathcal{Z}_\eta} \sum_{j \in I} \left( \hat{c}_{j}   + d_j z_j \right) 
\end{equation}
where the uncertainty set and reduced uncertainty set are rewritten  as
\begin{align}
 \mathcal{Z} \coloneqq \left \{  \bz \in [0,1]^n: \sum_{j=1}^n z_j \le \Gamma \right\} \text{ and }
   \mathcal{Z}_\eta = \left \{ \bz \in \mathcal{Z}: \min_{I \in \mathcal{I}} \sum_{j \in I} (\hat{c}_{j} + d_j z_j)  \ge \eta \right \}. \label{line:comb_unc_eta}
\end{align}

It is straightforward to apply our results from \S\ref{sec:mainresults} to the above problem setting. Indeed, the budget uncertainty set $\mathcal{U}$ is convex and compact, the feasible set $\mathcal{I}$ is finite, and the function $f(I,\bc) = \sum_{j \in I} c_j $ is linear in $\bc$ and bounded on $\mathcal{I}$ and $\mathcal{U}$, and so  Assumptions~\ref{ass:main} and \ref{ass:convex} are satisfied. 
Since Assumptions~\ref{ass:main} and \ref{ass:convex} are satisfied, and since \eqref{prob:lp} is a minimization problem, it follows from Proposition~\ref{prop:bounds} in \S\ref{sec:simplebounds} that the optimal value of the reduced robust optimization problem~\eqref{prob:lp_eta} is lower bounded by 
\begin{equation} \label{prob:lp_robust_exchange}
\begin{aligned}
\max_{\bz \in \mathcal{Z}} \min_{I \in \mathcal{I}} \sum_{j \in I} \left( \hat{c}_{j} + d_j z_j \right)
\end{aligned}
\end{equation}
and it follows from Corollary~\ref{cor:main} from \S\ref{sec:maintheorem} that the value of human expertise is equal to the gap between the  optimal value of  \eqref{prob:lp_robust_exchange} and the optimal value of the robust optimization problem
\begin{align}
\min_{I \in \mathcal{I}} \max_{\bz \in \mathcal{Z}} \sum_{j  \in I} \left(\hat{c}_j + d_j z_j \right ).  \label{prob:lp_robust}
\end{align}

\subsection{Structural Results for Budget Uncertainty Sets} \label{sec:structure:budget}

In this subsection, we develop structural results for the reduced robust optimization problem~\eqref{prob:lp_eta} with the budget uncertainty set $\mathcal{Z}$ from \eqref{line:comb_unc_eta}. The main results of this subsection are theoretical conditions under which the optimal value of the reduced robust optimization problem~\eqref{prob:lp_eta} is strictly  less than the optimal value of the robust optimization problem~\eqref{prob:lp_robust}. %
In particular, these theoretical results  reveal that {surprisingly  loose} beliefs about the optimal value of the nominal problem can lead to improved performance guarantees compared to those from robust optimization, and these theoretical findings  are corroborated numerically in \S\ref{sec:motivatingexample_shortest_path}.  All omitted proofs from the present subsection can be found in Appendix~\ref{app:proofs_structure}. 

Our analysis will make use of the \emph{non-robust} optimization problem, defined as \looseness=-1
\begin{align}
\min_{I \in \mathcal{I}} \sum_{j \in I} \hat{c}_j  \label{prob:nonrobust}
\end{align} 
The relationships between the non-robust problem~\eqref{prob:nonrobust} and the problems from \S\ref{sec:applications:optimizercurse:setting} are summarized as follows.\looseness=-1
\begin{lemma} \label{lem:lp_relationships}
The optimal value of the non-robust problem~\eqref{prob:nonrobust} is a lower bound on the optimal value of \eqref{prob:lp_robust_exchange}. Moreover, if $\eta \le \min_{I \in \mathcal{I}} \sum_{j \in I} \hat{c}_j  $, then $\mathcal{Z}_\eta = \mathcal{Z}$. 
\end{lemma}
\begin{proof}
It follows from the definition of the budget uncertainty set~\eqref{line:budget_uncertaintyset} that $\hat{\bc} \in \mathcal{U}$, which implies that the optimal value of  \eqref{prob:nonrobust} is a lower bound on the optimal value of \eqref{prob:lp_robust_exchange}.  It additionally follows from $\bd > \bzero$ and $\mathcal{Z} \subseteq [0,1]^n$ that $\sum_{j \in I} (\hat{c}_j + d_j z_j)   \ge \sum_{j \in I} \hat{c}_j$ for all $I \in \mathcal{I}$ and $\bz \in \mathcal{Z}$. This implies that $\mathcal{Z} = \mathcal{Z}_\eta$ for all $\eta \le \min_{I \in \mathcal{I}} \sum_{j  \in I} \hat{c}_j $. 
\end{proof}
The first part of Lemma~\ref{lem:lp_relationships}, combined with Proposition~\ref{prop:bounds}, implies that the optimal value of the non-robust problem~\eqref{prob:nonrobust} is a lower bound on the optimal value of the  reduced robust optimization problem~\eqref{prob:lp_eta}. The second part of Lemma~\ref{lem:lp_relationships}, combined with  Proposition~\ref{prop:eta}, shows that the reduced robust optimization problem~\eqref{prob:lp_eta} is well defined and can have an optimal value different from that of the original robust optimization problem~\eqref{prob:lp_robust} if and only if
\begin{align}
\eta \in \mathbb{H} \coloneqq \left (\min_{I \in \mathcal{I}} \sum_{j \in I} \hat{c}_j,  \max_{\bz \in \mathcal{Z}} \min_{I \in \mathcal{I}} \sum_{j \in I}  \left( \hat{c}_{j}  + d_j z_j \right) \right]. \label{line:H} 
\end{align}

Equipped with the above,  we are ready to develop our main results of \S\ref{sec:structure:budget}.  To motivate these results, recall that Theorem~\ref{thm:fundamental} shows that the optimal value of the reduced robust optimization problem~\eqref{prob:lp_eta} will be equal to the optimal value of the max-min problem~\eqref{prob:lp_robust_exchange} when $\eta$ is equal to the optimal value of \eqref{prob:lp_robust_exchange}.  This implies that if there is a gap between the optimal values of the min-max problem~\eqref{prob:lp_robust} and the max-min problem~\eqref{prob:lp_robust_exchange},  there must exist some `threshold' for $\eta \in \mathbb{H}$ after which the optimal value of the reduced robust optimization problem~\eqref{prob:lp_eta} becomes strictly less than the optimal value of the min-max problem~\eqref{prob:lp_robust}. Our main results of \S\ref{sec:structure:budget} characterize the location of this threshold by exploiting the structure of the budget uncertainty set.\looseness=-1

The first main result of this subsection, stated below as Corollary~\ref{cor:extreme}, establishes simple sufficient conditions for the optimal value of the reduced robust optimization problem~\eqref{prob:lp_eta} to be strictly less than the optimal value of the robust optimization problem~\eqref{prob:lp_robust} \emph{for all $\eta \in \mathbb{H}$}. Corollary~\ref{cor:extreme} is a special case of  a more general result about the worst-case performance of fixed solutions under the reduced budget uncertainty set, stated below as Theorem~\ref{thm:extreme}. In the following, we define $\text{supp}(\bz) \coloneqq \{j: z_j > 0 \}$.\looseness=-1

\begin{theorem} \label{thm:extreme}
Let $\tilde{I} \in \mathcal{I}$. Then
\begin{align}
\max_{\bz \in \mathcal{Z}_\eta}   \sum_{j \in \tilde{I}}  \left( \hat{c}_j + d_j z_j  \right) < \max_{\bz \in \mathcal{Z}}   \sum_{j \in \tilde{I}}  \left( \hat{c}_j + d_j z_j  \right)  \quad \forall \eta \in \mathbb{H} \label{line:unnec}
\end{align}
if and only if for every optimal solution $\tilde{\bz}$ of $\max_{\bz \in \mathcal{Z}} \sum_{j \in \tilde{I}} (\hat{c}_j + d_j z_j)$, there exists an optimal solution $I^* = I^*(\tilde{\bz})$ for the non-robust problem~\eqref{prob:nonrobust} that satisfies $\textnormal{supp}(\tilde{\bz}) \cap I^* = \emptyset$.\looseness=-1
\end{theorem}
In words,  Theorem~\ref{thm:extreme} says that a fixed solution $\tilde{I} \in \mathcal{I}$ satisfies \eqref{line:unnec}  if and only if every worst-case realization from the original uncertainty set $\tilde{\bz} \in \argmax_{\bz \in \mathcal{Z}} \sum_{j \in \tilde{I}} (\hat{c}_j + d_j z_j)$ does not affect the objective value of at least one optimal solution of the non-robust problem~\eqref{prob:nonrobust}.  Corollary~\ref{cor:extreme} extends Theorem~\ref{thm:extreme} to the optimal value of the reduced robust optimization problem~\eqref{prob:lp_eta} and replaces the `if and only if' claim from Theorem~\ref{thm:extreme} with simpler sufficient conditions.\looseness=-1

\begin{corollary}\label{cor:extreme}
We have
\begin{align*}
\min_{I \in \mathcal{I}} \max_{\bz \in \mathcal{Z}_\eta}   \sum_{j \in I}  \left( \hat{c}_j + d_j z_j  \right) < \min_{I \in \mathcal{I}} \max_{\bz \in \mathcal{Z}}   \sum_{j \in I}  \left( \hat{c}_j + d_j z_j  \right)  \quad \forall \eta \in \mathbb{H}
\end{align*}
if there exists an optimal solution  $I^{\text{RO}}$ of the robust optimization problem~\eqref{prob:lp_robust} that satisfies either of the following two conditions:
\begin{enumerate}[(a)]
\item $|I^{\text{RO}}| \ge \Gamma$ and there exists an optimal solution $I^*$ for the non-robust problem~\eqref{prob:nonrobust} that satisfies $I^{\text{RO}} \cap I^* = \emptyset$.\looseness=-1 \label{suffcond:a}
\item The problem $\max_{\bz \in \mathcal{Z}} \sum_{j \in I^{\text{RO}}} (\hat{c}_j + d_j z_j)$ has a unique optimal solution $\tilde{\bz}$, and there is an optimal solution  $I^*$ for the non-robust problem~\eqref{prob:nonrobust} that satisfies $\textnormal{supp}(\tilde{\bz}) \cap I^* = \emptyset$. \label{suffcond:b}
\end{enumerate}
\end{corollary}

Informally, the above corollary reveals that the conservatism of the reduced robust optimization problem~\eqref{prob:lp_eta} is driven by the {similarity} of the optimal solutions of the original robust optimization problem~\eqref{prob:lp_robust} and the non-robust problem~\eqref{prob:nonrobust}. More formally,   Corollary~\ref{cor:extreme} says that if either of the sufficient conditions from Corollary~\ref{cor:extreme} is satisfied, then any $\eta \in \mathbb{H}$  will make the optimal value of the reduced robust optimization problem~\eqref{prob:lp_eta} strictly less than the optimal value of the original robust optimization problem~\eqref{prob:lp_robust}.
 The first sufficient condition \ref{suffcond:a} is that there exists  an optimal solution of the original robust optimization problem~\eqref{prob:lp_robust} that has a sufficiently large support and is disjoint from an optimal solution for  the non-robust problem~\eqref{prob:nonrobust}. The second sufficient condition~\ref{suffcond:b} is that the optimal solution of the original robust optimization problem~\eqref{prob:lp_robust} has a unique worst-case realization from the uncertainty set $\mathcal{Z}$, and that this worst-case realization does not affect the cost coefficients of an optimal solution for the non-robust problem~\eqref{prob:nonrobust}. We remark that the uniqueness of an optimal solution of $\max_{\bz \in \mathcal{Z}} \sum_{j \in I^{\text{RO}}} (\hat{c}_j + d_j z_j)$ is guaranteed under the assumptions of Corollary~\ref{cor:extreme}  if $| I^{\text{RO}}| \ge \Gamma$ and the values of  $d_j$ for $j \in I^{\text{RO}}$  are distinct. 

It is worthwhile to make a couple of remarks about the practical significance of  Theorem~\ref{thm:extreme} and Corollary~\ref{cor:extreme}. First, we believe that the sufficient conditions of Corollary~\ref{cor:extreme}, or the more general condition from Theorem~\ref{thm:extreme}, are relatively mild and may be satisfied in realistic applications. To illustrate this, we show in \S\ref{sec:motivatingexample_shortest_path} that these conditions can be satisfied in a numerical experiment from \citet[\S6.3]{MelvynSim2003}.  A second takeaway from Corollary~\ref{cor:extreme} is a new insight about the conservatism of robust combinatorial optimization with budget uncertainty sets. Specifically, our analysis shows that if the optimal solution for the robust optimization problem~\eqref{prob:lp_robust} is much different than the optimal solution of the non-robust problem~\eqref{prob:nonrobust} (e.g., if either condition~\ref{suffcond:a} or \ref{suffcond:b} is satisfied), then Corollary~\ref{cor:extreme} reveals that the worst-case realizations of $\bz \in \mathcal{Z}$ for the optimal solution of the robust optimization problem~\eqref{prob:lp_robust} are those  that would allow for high-quality best-case performance. This insight is formalized by the following proposition.\looseness=-1

\begin{proposition} \label{prop:broader}
If $I^{\text{RO}}$ is an optimal solution of the robust optimization problem~\eqref{prob:lp_robust} that satisfies conditions~\ref{suffcond:a} or \ref{suffcond:b}, then every worst-case realization  
\begin{align*}
\tilde{\bz} \in \argmax_{\bz \in \mathcal{Z}}  \sum_{j \in I^{\text{RO}}}  \left( \hat{c}_j + d_j z_j  \right)
\end{align*}
satisfies
\begin{align*}
 \min_{I \in \mathcal{I}} \sum_{j \in I} (\hat{c}_{j} + d_j \tilde{z}_j) =   \min_{I \in \mathcal{I}} \sum_{j \in I} \hat{c}_j 
\end{align*}
\end{proposition} 

As far as we can tell, Proposition~\ref{prop:broader} offers a new insight within the robust optimization literature. Specifically, it shows that the parameters that are worst for the optimal solution of the robust optimization problem are parameters under which a decision maker with full information could still achieve the smallest possible optimal cost. 
Therefore, if the decision maker believes the  optimal value of the nominal problem~\eqref{prob:lp} is likely to be strictly greater than the optimal value of the non-robust problem~\eqref{prob:nonrobust}, then the parameters that are worst for the optimal solution of the robust optimization problem  are ruled out. 

This proposition is particularly relevant in the context of the present paper, as it suggests that nominal curves for optimal solutions for \eqref{prob:lp_robust} may be practically informative. Indeed, suppose that a decision maker computes an optimal solution $I^{\text{RO}}$ for the original robust optimization problem~\eqref{prob:lp_robust} but is hesitant  about deploying that decision out of concern that the performance of $I^{\text{RO}}$ on non-worst-case realizations of $\bz \in \mathcal{Z}$ may be nearly as poor as the worst-case performance of $I^{\text{RO}}$. Proposition~\ref{prop:broader} suggests that this concern can potentially be alleviated through the nominal curve of $I^{\text{RO}}$,  if the decision maker believes that  the optimal value of the nominal problem~\eqref{prob:lp} is likely to be strictly higher than the optimal value of the non-robust problem~\eqref{prob:nonrobust}. We elaborate on this  takeaway from Proposition~\ref{prop:broader} in numerical experiments in \S\ref{sec:motivatingexample_shortest_path}.

We conclude \S\ref{sec:structure:budget} by developing results about the optimal value of the reduced robust optimization problem~\eqref{prob:lp_eta} in settings where the conditions from Theorem~\ref{thm:extreme} and Corollary~\ref{cor:extreme} are not satisfied.  We develop these results by drawing connections between the reduced uncertainty set $\mathcal{Z}_\eta$ and the \emph{hitting set problem}. 
To begin, let the collection of near-optimal  solutions for the non-robust problem~\eqref{prob:nonrobust} be denoted by \looseness=-1
\begin{align}
\mathcal{I}_\eta \coloneqq \left \{ I \in \mathcal{I}:  \eta - \sum_{j \in I} \hat{c}_j  > 0 \right \}.\label{prob:nonrobust_old}
\end{align}
It follows from the above discussion  that the above set is nonempty for all $\eta \in \mathbb{H}$.  In view of the above notation, the following  lemma shows  that the constraints in the reduced uncertainty set $\mathcal{Z}_\eta$ are driven by the elements of $\mathcal{I}_\eta$. 

\begin{lemma} \label{lem:reduced_comb}
The reduced uncertainty set satisfies
\begin{align*}
\mathcal{Z}_\eta &= \left \{  \bz \in \mathcal{Z}:   \sum_{j \in I}  d_j z_j   \ge \eta -  \sum_{j \in I}  \hat{c}_j  \; \forall I \in \mathcal{I}_\eta \right  \}.
\end{align*}
\end{lemma}

The above lemma shows that as $\eta$ increases, the \emph{number} as well as the \emph{tightness} of the constraints in the reduced uncertainty set $\mathcal{Z}_\eta$ increase.
In particular, this makes it possible to derive sufficient conditions for when the optimal value of \eqref{prob:lp_eta} is less than that of \eqref{prob:lp_robust} based on the combinatorial structure of $\mathcal{I}_\eta$. To that end, let the hitting set number of $\mathcal{I}_\eta$ be defined as the minimum cardinality of a set that intersects each element of $\mathcal{I}_\eta$:\looseness=-1
\begin{align}
\textsc{HittingSetNum}(\mathcal{I}_\eta) &\coloneqq
 \min_{\bz \in \{0,1\}^n} \left \{ \sum_{j=1}^n z_j: \sum_{j \in I} z_j \ge 1 \; \text{for all } I \in \mathcal{I}_\eta  \right \}.  \label{line:hittingset}
\end{align}
Equipped with the above terminology, we obtain the following sufficient condition for the optimal value of the reduced robust optimization problem~\eqref{prob:lp_eta} to be strictly less than the optimal value of the robust optimization problem~\eqref{prob:lp_robust}  that can hold when the conditions from Theorem~\ref{thm:extreme} and Corollary~\ref{cor:extreme} are not satisfied. 
\begin{proposition} \label{prop:hitting}
Suppose there exists an optimal solution $I^{\text{RO}}$ of the robust optimization problem~\eqref{prob:lp_robust} for which $\max_{\bz \in \mathcal{Z}} \sum_{j \in I^{\text{RO}}} (\hat{c}_j + d_j z_j)$ has a unique solution. If $\eta \in \mathbb{H}$ and $\textsc{HittingSetNum}(\mathcal{I}_\eta) > \Gamma$, then 
\begin{align*}
\min_{I \in \mathcal{I}} \max_{\bz \in \mathcal{Z}_\eta}   \sum_{j \in I}  \left( \hat{c}_j + d_j z_j  \right) < \min_{I \in \mathcal{I}} \max_{\bz \in \mathcal{Z}}   \sum_{j \in I}  \left( \hat{c}_j + d_j z_j  \right)
\end{align*}
\end{proposition}

The above proposition gives a condition different from Corollary~\ref{cor:extreme} 
for when the optimal value of the reduced robust optimization problem~\eqref{prob:lp_eta} will be strictly less than the optimal value of the robust optimization problem~\eqref{prob:lp_robust}. The main idea is that if $\textsc{HittingSetNum}(\mathcal{I}_\eta) > \Gamma$, then it must be the case that $| \text{supp}(\bz)| > \Gamma$ for all $\bz \in \mathcal{Z}_\eta$.  This implies that none of the extreme points of $\mathcal{Z}$ are contained in $\mathcal{Z}_\eta$, and thus the unique optimal solution of the linear program  $\max_{\bz \in \mathcal{Z}} \sum_{j \in I^{\text{RO}}} (\hat{c}_j + d_j z_j)$ is not contained in $\mathcal{Z}_\eta$.  Proposition~\ref{prop:hitting} thus shows that a strictly positive value of human expertise can arise from the non-robust problem~\eqref{prob:nonrobust} having a diversity of near-optimal solutions.  %

\subsection{Numerical Experiment} \label{sec:motivatingexample_shortest_path}

  \begin{figure}[t]
  \centering
    \includegraphics[width=0.8\textwidth]
      {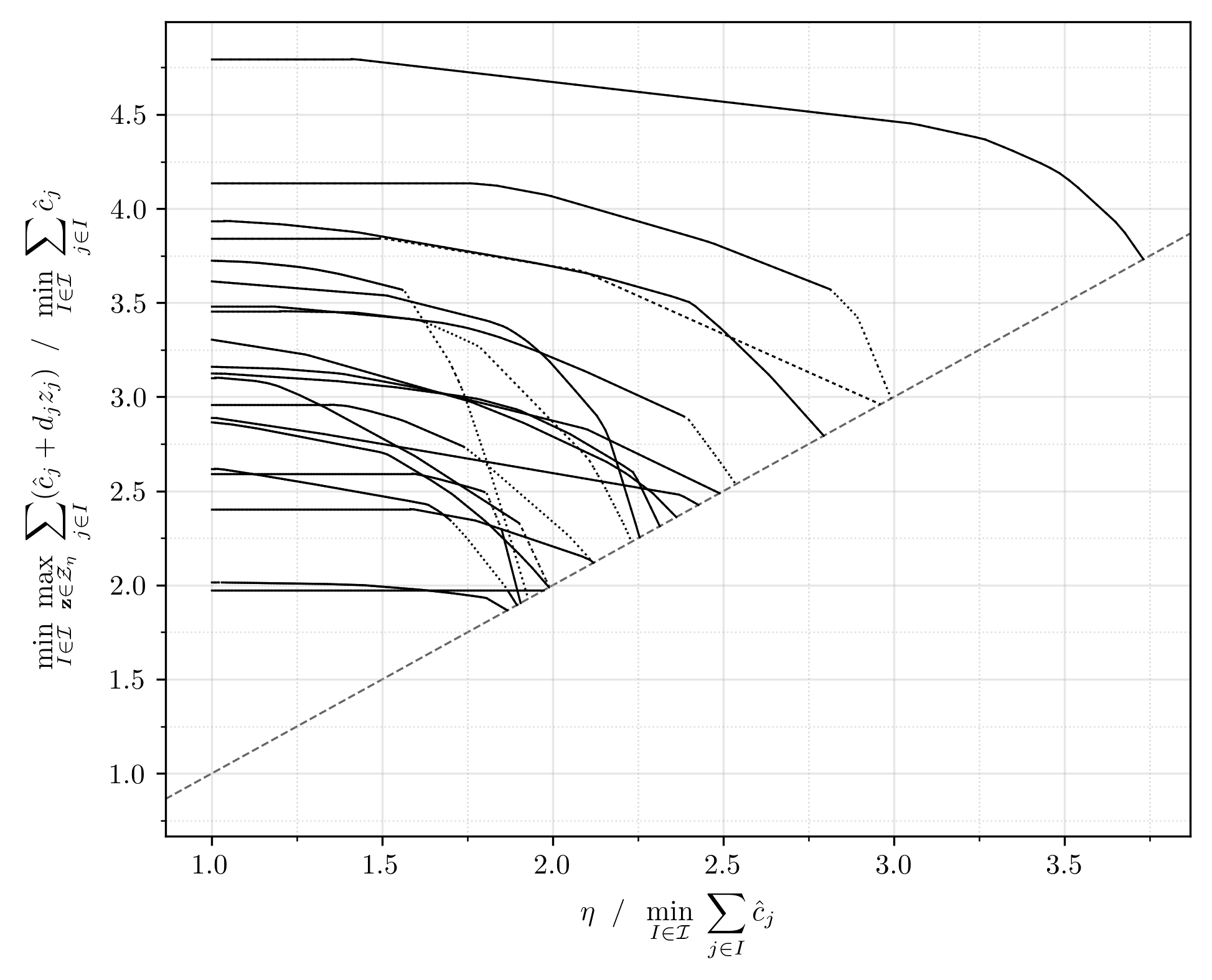}
    \caption{Results of numerical experiments from
      \S\ref{sec:motivatingexample_shortest_path}.}
    \label{fig:shortest_path_plot}
  \end{figure}
  
Equipped with the results from \S\ref{sec:structure:budget}, we illustrate the value of human expertise and nominal curves through numerical experiments on a shortest path problem with a budget uncertainty set. Our experiments closely follow the setup from \citet[\S6.3]{MelvynSim2003}. Each instance of the experiment involves generating a random graph with nodes in $[0,1]^2$, estimating edge costs $\hat{\bc}$ as the Euclidean distance between each pair of nodes in the graph, and using a budget uncertainty set anchored at $\hat{\bc}$ of the form given in \eqref{line:budget_uncertaintyset}.\looseness=-1

In greater detail, each instance of the experiment consists of a randomly generated graph with $| \mathcal{N}| = 60$ nodes and $| \mathcal{E}|=295$ edges in a two-dimensional Euclidean space. The start node is assigned to coordinate $(0,0)$, the target node is placed at coordinate $(1,1)$, and the remaining nodes are assigned uniformly at random over $[0,1]^2$. The edges are chosen uniformly at random from the set of all pairs of nodes. The set of feasible decisions $\mathcal{I} \subset 2^{\mathcal{E}}$ is the set of all paths from the start node to the target node. The edge cost $\hat{c}_{i,j}$ is calculated as the Euclidean distance between nodes $i$ and $j$. Each $d_{i,j} \ge 0$ denotes the maximum deviation of the cost on edge $(i,j)$ from its prediction $\hat{c}_{i,j}$, and $d_{i,j}$ is set to be equal to $\gamma_{i,j} \hat{c}_{i,j}$ , where $\gamma_{i,j}$ is uniformly distributed in $[0, 8]$. The budget parameter is $\Gamma = 3$.  

Our  numerical experiment consists of 20 randomly generated instances. In each instance, we solve the robust optimization problem~\eqref{prob:lp_robust} and the non-robust problem~\eqref{prob:nonrobust} once,  and the reduced robust optimization problem~\eqref{prob:lp_eta} once for each of 101 different values of $\eta$ equally spaced over the closure of $\mathbb{H}$. We solve the reduced robust optimization problem~\eqref{prob:lp_eta} for each such $\eta$ using the mixed-integer linear programming formulation from Proposition~\ref{prop:reform:simple} in \S\ref{sec:algorithm:reform}. The results of the experiments are shown in Figure~\ref{fig:shortest_path_plot} and Table~\ref{tab:paths}.\looseness=-1

Figure~\ref{fig:shortest_path_plot} plots a line for each of the 20 randomly generated  instances of the shortest path problem. For each instance, the corresponding line in Figure~\ref{fig:shortest_path_plot} shows the optimal value of the reduced robust optimization problem~\eqref{prob:lp_eta} as a function of $\eta$, where both quantities are normalized by the optimal value of the non-robust problem~\eqref{prob:nonrobust}. In other words, Figure~\ref{fig:shortest_path_plot} shows the minimization equivalent of the upper bound  curves~\eqref{prob:robust_eta_set} from \S\ref{sec:managerialexpertise:generating_policies}.  The line corresponding to an instance is solid if the optimal solution for the robust optimization problem~\eqref{prob:lp_robust} is an optimal solution for the reduced robust optimization problem~\eqref{prob:lp_eta}, and the line is dotted otherwise. The diagonal dashed line shows $y = x$.  Consistent with Corollary~\ref{cor:main}, we observe from Figure~\ref{fig:shortest_path_plot} for each instance that when $\eta$ reaches its maximum value, the optimal value of \eqref{prob:lp_eta} is equal to $\eta$. 

\begin{table}[t]
\centering
\caption{Optimal solutions for numerical experiments from \S\ref{sec:motivatingexample_shortest_path}.}
\label{tab:paths}
{\footnotesize
\begin{tabular}{c l l c c}
\toprule
\multicolumn{1}{c}{} & \multicolumn{1}{c}{Optimal Solution of \eqref{prob:nonrobust}} & \multicolumn{1}{c}{Optimal Solution of \eqref{prob:lp_robust}} & \multicolumn{1}{c}{Cond~\ref{suffcond:a}} & \multicolumn{1}{c}{Cond~\ref{suffcond:b}} \\
\midrule
1 & $1 \to 31 \to 41 \to 33 \to 60$ & $1 \xrightarrow{\star} 31 \to 41 \xrightarrow{\star} 12 \to 34 \to 43 \xrightarrow{\star} 60$ &  &  \\
2 & $1 \to 12 \to 34 \to 60$ & $1 \xrightarrow{\star} 58 \to 2 \xrightarrow{\star} 27 \xrightarrow{\star} 34 \to 60$ &  & $\checkmark$ \\
3 & $1 \to 23 \to 49 \to 17 \to 34 \to 14 \to 60$ & $1 \xrightarrow{\star} 8 \to 2 \xrightarrow{\star} 26 \to 56 \to 37 \to 21 \to 14 \xrightarrow{\star} 60$ &  &  \\
4 & $1 \to 16 \to 40 \to 60$ & $1 \to 16 \to 8 \xrightarrow{\star} 48 \xrightarrow{\star} 38 \to 28 \to 17 \xrightarrow{\star} 60$ &  & $\checkmark$ \\
5 & $1 \to 36 \to 44 \to 60$ & $1 \xrightarrow{\star} 36 \xrightarrow{\star} 44 \xrightarrow{\star} 60$ &  &  \\
6 & $1 \to 52 \to 60$ & $1 \xrightarrow{\star} 52 \xrightarrow{\star} 46 \to 33 \to 58 \xrightarrow{\star} 49 \to 10 \to 60$ &  &  \\
7 & $1 \to 47 \to 44 \to 60$ & $1 \to 38 \xrightarrow{\star} 3 \to 39 \to 59 \xrightarrow{\star} 44 \xrightarrow{\star} 60$ &  &  \\
8 & $1 \to 2 \to 60$ & $1 \xrightarrow{\star} 45 \to 46 \xrightarrow{\star} 2 \xrightarrow{\star} 60$ &  &  \\
9 & $1 \to 25 \to 35 \to 38 \to 60$ & $1 \xrightarrow{\star} 16 \to 43 \xrightarrow{\star} 30 \to 41 \to 38 \xrightarrow{\star} 60$ &  &  \\
10 & $1 \to 60$ & $1 \xrightarrow{\star} 45 \xrightarrow{\star} 40 \xrightarrow{\star} 60$ & $\checkmark$ & $\checkmark$ \\
11 & $1 \to 44 \to 36 \to 41 \to 60$ & $1 \to 44 \xrightarrow{\star} 36 \to 41 \xrightarrow{\star} 9 \xrightarrow{\star} 2 \to 60$ &  &  \\
12 & $1 \to 56 \to 59 \to 7 \to 41 \to 60$ & $1 \xrightarrow{\star} 56 \to 59 \to 7 \xrightarrow{\star} 13 \xrightarrow{\star} 60$ &  &  \\
13 & $1 \to 41 \to 37 \to 60$ & $1 \xrightarrow{\star} 31 \xrightarrow{\star} 60$ &  &  \\
14 & $1 \to 51 \to 22 \to 23 \to 60$ & $1 \xrightarrow{\star} 26 \to 38 \to 36 \xrightarrow{\star} 14 \to 23 \xrightarrow{\star} 60$ &  &  \\
15 & $1 \to 23 \to 60$ & $1 \xrightarrow{\star} 26 \xrightarrow{\star} 60$ &  &  \\
16 & $1 \to 46 \to 14 \to 32 \to 60$ & $1 \to 34 \xrightarrow{\star} 46 \xrightarrow{\star} 9 \xrightarrow{\star} 60$ & $\checkmark$ & $\checkmark$ \\
17 & $1 \to 36 \to 58 \to 60$ & $1 \xrightarrow{\star} 54 \xrightarrow{\star} 23 \xrightarrow{\star} 60$ & $\checkmark$ & $\checkmark$ \\
18 & $1 \to 38 \to 8 \to 24 \to 54 \to 18 \to 60$ & $1 \xrightarrow{\star} 38 \to 8 \to 24 \xrightarrow{\star} 54 \to 18 \xrightarrow{\star} 60$ &  &  \\
19 & $1 \to 60$ & $1 \xrightarrow{\star} 60$ &  &  \\
20 & $1 \to 60$ & $1 \xrightarrow{\star} 51 \xrightarrow{\star} 55 \to 50 \xrightarrow{\star} 60$ & $\checkmark$ & $\checkmark$ \\
\bottomrule
\end{tabular}
}
\end{table}

Table~\ref{tab:paths} shows the optimal solution for the non-robust problem~\eqref{prob:nonrobust} and the optimal solution for the robust optimization problem~\eqref{prob:lp_robust} for each of the random instances. For the optimal solution $I^{\text{RO}}$ of the robust optimization problem~\eqref{prob:lp_robust}, a star is indicated above each edge in Table~\ref{tab:paths} that is in the support of the optimal solution of $\max_{\bz \in \mathcal{Z}} \sum_{j \in I^{\text{RO}}} (\hat{c}_j + d_j z_j)$. Table~\ref{tab:paths} uses checkmarks to indicate whether Conditions~\ref{suffcond:a} and \ref{suffcond:b} from Corollary~\ref{cor:extreme} are satisfied.

There are several takeaways from Figure~\ref{fig:shortest_path_plot} and Table~\ref{tab:paths}. First, we observe that there is a gap between the optimal values of \eqref{prob:lp_robust} and \eqref{prob:lp_robust_exchange}---and thus a value of human expertise that is strictly positive---in 19 out of the 20 random instances. Moreover,  when there is a gap between the optimal values of \eqref{prob:lp_robust} and \eqref{prob:lp_robust_exchange}, the gap can be considerable, with an average ratio of 1.389 between the optimal values of \eqref{prob:lp_robust} and \eqref{prob:lp_robust_exchange} over the 19 instances with a strict gap.  Second, Table~\ref{tab:paths} shows that six of the 20 random instances satisfied at least one of the sufficient conditions from Corollary~\ref{cor:extreme}, thereby implying that the optimal value of the reduced robust optimization problem~\eqref{prob:lp_eta} is strictly less than the optimal value of the robust optimization problem~\eqref{prob:lp_robust} for all $\eta \in \mathbb{H}$. Third, in many of the curves in Figure~\ref{fig:shortest_path_plot}, the optimal solution for the robust optimization problem~\eqref{prob:lp_robust} remained optimal for the reduced robust optimization problem~\eqref{prob:lp_eta} for many values of $\eta$, including those for which the optimal value of \eqref{prob:lp_eta} is strictly less than that of \eqref{prob:lp_robust}. This is consistent with Proposition~\ref{prop:broader} and demonstrates that nominal curves of optimal solutions for traditional robust combinatorial optimization problems~\eqref{prob:lp_robust} with  budget uncertainty sets can reveal non-trivial performance guarantees.\looseness=-1 %


\section{Finding Policies and Controlling Disappointment} \label{sec:practical}

To use the developments from this paper in real-world applications, we recommend using the approach described in  \S\ref{sec:managerialexpertise:generating_policies} for generating a menu of policies to offer to a decision  maker. That is, we recommend generating a menu of policies by solving the reduced robust optimization problem~\eqref{prob:robust_eta} over a range of values of $\eta$ and comparing those policies to one another through their nominal curves~\eqref{line:wc_eta_set}. By comparing the nominal curves of the policies, the decision maker can apply their own judgment about which scenarios of the nominal problem they believe are plausible and then select one of the policies.

In this section, we  facilitate the practical deployment of the approach from \S\ref{sec:managerialexpertise:generating_policies} in several ways.  In \S\ref{sec:applications:optimizercurse:algorithms}, we propose two  computational methods that can be used to solve the reduced robust optimization problem~\eqref{prob:robust_eta} and generate the menu of policies. In \S\ref{appx:disappointment}, we propose two alternative approaches to generating  a menu of policies that may be useful when the decision maker does not find the nominal curves of the policies obtained from the approach in \S\ref{sec:managerialexpertise:generating_policies} to be desirable. \looseness=-1 %

\subsection{Algorithms} \label{sec:applications:optimizercurse:algorithms}
In this section, we propose two computational methods for solving the reduced robust optimization problem~\eqref{prob:robust_eta}. The first  is an exact reformulation for a class of mixed-integer linear-fractional programs. The second is more general and based on the cutting plane method. 

\subsubsection{Exact Reformulations} \label{sec:algorithm:reform}
In some applications, the reduced robust optimization problem~\eqref{prob:robust_eta} can be solved by reformulating it as a compact mixed-integer linear program.   Below, we present such reformulations for the following applications when the uncertainty set is a nonempty bounded polyhedron.%

\begin{example} \label{example:comb} Combinatorial optimization problems of the form 
\begin{align*}
\max_{\bx \in \{0,1\}^n: \bba \bx \le \bb}\;   \sum_{j=1}^n \bar{\theta}_j  x_j
\end{align*}
where $\bba \in \R^{m \times n}$ is totally unimodular, $\bb$ is integral, and the constraints of the form $\bzero \le \bx \le \bone$ have without loss of generality been embedded within $\bba \bx \le \bb$. %
Examples include shortest path, bipartite matching, and linear assignment problems. 
\end{example}
\begin{example} \label{example:mnl} Cardinality-constrained assortment optimization problems under  the multinomial logit model, \looseness=-1
\begin{align}
\max_{S \subseteq \{1,\ldots,n\}: |S| \le k}  \sum_{i \in S} \frac{r_i \bar{\theta}_i}{1 + \sum_{j \in S} \bar{\theta}_j}\label{prob:mnl_eta}
\end{align} 
where $r_1,\ldots,r_n > 0$ and $\bar{\theta}_1,\ldots,\bar{\theta}_n \ge 0$. Note that \eqref{prob:mnl_eta} does not satisfy Assumption~\ref{ass:convex} because $\btheta \mapsto  \sum_{i \in S} {r_i {\theta}_i}/{(1 + \sum_{j \in S} {\theta}_j)}$ is a nonconvex function in general. Even though the conditions of Corollary~\ref{cor:main} are not satisfied in this setting, the value of human expertise can still be strictly positive; %
we show an example of this at the end of \S\ref{sec:algorithm:reform}.\looseness=-1

\end{example}

To develop a mixed-integer linear programming reformulation of \eqref{prob:robust_eta} in the above examples, we focus on the more general nominal problem
\begin{align*}
\max_{\bx \in \mathcal{X}} \frac{\bx^\intercal \bbu \bar{\btheta} + \mu}{\bx^\intercal \bbl \bar{\btheta} + \lambda}
\end{align*}
where $\mathcal{X} = \{\bx \in \{0,1\}^n: \bba \bx \le \bb \} \neq \emptyset$  satisfies $\conv(\mathcal{X}) = \{ \bx \in \R^n: \bba \bx \le \bb \}$,  the uncertainty set  is a bounded polyhedron $\mathcal{U} = \{\btheta \in \R^p: \bbd \btheta \ge \bg \} \neq \emptyset$, and the dimensions are $\bba \in \R^{m \times n}$, $\bb \in \R^m$, $\bbd \in \R^{q \times p}$, $\bg \in \R^q$,  $\bbu, \bbl \in \R^{n \times p}$, and $\mu,\lambda \in \R$. We observe that Examples~\ref{example:comb} and \ref{example:mnl} are special cases of this setting.\looseness=-1

In what follows, we show that the reduced robust optimization problem~\eqref{prob:robust_eta} corresponding to the above setting, denoted by  \looseness=-1
\begin{align}
\max_{\bx \in \mathcal{X}}\min_{\btheta \in \mathcal{U}_\eta}  \frac{\bx^\intercal \bbu {\btheta} + \mu}{\bx^\intercal \bbl {\btheta} + \lambda},\label{prob:fraction_eta}
\end{align}
can be reformulated as a mixed-integer linear program of polynomial size. 
The key idea behind the reformulation of \eqref{prob:fraction_eta} is applying strong duality twice:  once to construct a polynomial-size extended formulation of  the reduced uncertainty set $\mathcal{U}_\eta$, and once to dualize the inner problem of the reduced robust optimization problem~\eqref{prob:fraction_eta}. We first state our reformulation of \eqref{prob:fraction_eta} for the special case of Example~\ref{example:comb}, followed by the reformulation for the general case; the proofs of both of these propositions are found in Appendix~\ref{appx:messy_reforms}. 
\begin{proposition} \label{prop:reform:simple}
If $\mathcal{U}_\eta$ is nonempty, then 
    \begin{align*}
\max_{\bx \in \mathcal{X}} \min_{\btheta \in \mathcal{U}_\eta} \btheta^\intercal \bx = \left[ \begin{aligned}
&\underset{ \bx \in \{0,1\}^n, \bpsi \in \R^q,  \sigma \in \R}{\textnormal{maximize}} &&   \bg^\intercal \bpsi - \eta \sigma\\
&\textnormal{subject to}&&     \bba \bx \le \bb\\
 &&&     \bba \left( \bbd^\intercal \bpsi - \bx \right) \le \bb \sigma\\
 &&&  \bpsi \ge \bzero,\sigma \ge 0
    \end{aligned}\right] 
    \end{align*}
\end{proposition}
\begin{proposition} \label{prop:reform:general}
If $\mathcal{U}_\eta$ is nonempty, $\bx^\intercal \bbl {\btheta} + \lambda > 0$ for all $\bx \in \mathcal{X}$ and $\btheta \in \mathcal{U}$, and the optimal value of \eqref{prob:fraction_eta} lies in $[\ubar{t},\bar{t}]$ for $0 \le \ubar{t} \le \bar{t}$, then\looseness=-1
    \begin{align*}
\max_{\bx \in \mathcal{X}}\min_{\btheta \in \mathcal{U}_\eta}  \frac{\bx^\intercal \bbu {\btheta} + \mu}{\bx^\intercal \bbl {\btheta} + \lambda} = 
\left[ \begin{aligned}
&\underset{ \substack{\bx \in \{0,1\}^n, t \in \R, \bpsi \in \R^q,\\ \by \in \R^n, \sigma \in \R, \bz \in \R^n}}{\textnormal{maximize}} &&    t\\
&\textnormal{subject to}&&
t \lambda - \mu \le \bg^\intercal \bpsi - (\eta \lambda - \mu) \sigma\\
    &&&  \bbd^\intercal \bpsi - \left( \bbu - \eta \bbl \right)^\intercal \by
        = \bbu^\intercal \bx  -  \bbl^\intercal \bz  \\
   &&& \ubar{t} \le t \le \bar{t}\\
    &&&\ubar{t} x_i \le z_i \le \bar{t}x_i && \forall i \in \{1,\ldots,n\}\\
  &&&  t - \bar{t} (1-x_i) \le z_i   \le t - \ubar{t} (1-x_i) && \forall i \in \{1,\ldots,n\}    \\
      &&& \bba \bx \le \bb\\
 &&&     \bba \by \le \bb \sigma\\
 &&&  \bpsi \ge \bzero,\sigma \ge 0
    \end{aligned}\right].
    \end{align*}
\end{proposition}

The mixed-integer programming formulation from Proposition~\ref{prop:reform:simple}  is the reformulation of \eqref{prob:fraction_eta} for the special case of $n = p$, $\bbl = \bzero$, $\bbu$ is the identity matrix, $\lambda = 1$, and $\mu = 0$. As such, the formulation from Proposition~\ref{prop:reform:simple} corresponds to Example~\ref{example:comb}, and this formulation is used in the numerical experiments from  \S\ref{sec:motivatingexample_shortest_path}.\footnote{The numerical experiments in \S\ref{sec:motivatingexample_shortest_path} focus on the min-max formulation of the problem, which can be obtained from Proposition~\ref{prop:reform:simple} by negating the objective function. } The mixed-integer programming formulation from Proposition~\ref{prop:reform:general} is general and requires that we have a known nonnegative lower bound $\ubar{t}$ and upper bound $\bar{t}$ on the optimal value of \eqref{prob:fraction_eta}. In the context of Example~\ref{example:mnl}, the lower bound can be chosen as $\$0$ or as the worst-case expected revenue of a heuristic assortment, and the upper bound can be chosen as the revenue of the most expensive product.\looseness=-1

We conclude \S\ref{sec:algorithm:reform} by providing an instance of Example~\ref{example:mnl} in which the value of human expertise is strictly positive. The example is useful for two reasons. First, it shows that the exact reformulation from Proposition~\ref{prop:reform:general} is not vacuous, in the sense that the reduced robust optimization problem~\eqref{prob:robust_eta} can differ from the robust optimization problem~\eqref{prob:robust} and is thus worthwhile to solve.   Second, it demonstrates that the value of human expertise can sometimes be strictly positive in settings where the conditions of Corollary~\ref{cor:main} are not satisfied, although this is not the case in general as shown by  Example~\ref{example:nonconvex} in \S\ref{sec:maintheorem}. %
\begin{example}
Let $n = 2$, $k = 1$, $r_1 = r_2 = 1$, and $\mathcal{U} = \{\btheta \in \R^2: \sum_{i=1}^2 \theta_i = 3, \btheta \ge \bone \}$. Then
\begin{align*}
\max_{S: |S| \le 1} \min_{\btheta \in \mathcal{U}} \sum_{i \in S} \frac{r_i {\theta}_i}{1 + \sum_{j \in S} {\theta}_j} &= \min_{\btheta  \ge \bone:\bone^\intercal \btheta = 3} \frac{\theta_1}{1 + \theta_1} = \frac{1}{2}, \text{ and }\\
 \min_{\btheta \in \mathcal{U}} \max_{S: |S| \le 1}  \sum_{i \in S} \frac{r_i {\theta}_i}{1 + \sum_{j \in S} {\theta}_j} &= \frac{\sfrac{3}{2}}{1 + \sfrac{3}{2}} = \frac{3}{5}.
\end{align*}
We observe that if $\eta = \sfrac{3}{5}$, then
\begin{align*}
\max_{S: |S| \le 1} \sum_{i \in S} \frac{r_i \theta_i}{1 + \sum_{j \in S} \theta_j} \le \eta &\implies  \frac{\theta_i}{1 + \theta_i} \le  \frac{3}{5} \quad  \forall i \in \{1,2\}\\
 &\implies \theta_i \le \frac{3}{2} \quad  \forall i \in \{1,2\},
\end{align*}
and so it follows from the definition of $\mathcal{U}$ that $\mathcal{U}_{\sfrac{3}{5}} = \{(\sfrac{3}{2}, \sfrac{3}{2})\}$. We thus conclude for the case of $\eta = \sfrac{3}{5}$ that\looseness=-1
\begin{align*}
\max_{S: |S| \le 1} \min_{\btheta \in \mathcal{U}_\eta} \sum_{i \in S} \frac{r_i {\theta}_i}{1 + \sum_{j \in S} {\theta}_j}  = \max_{S: |S| \le 1}  \sum_{i \in S}  \frac{\sfrac{3}{2}}{1 + \sfrac{3}{2}}  = \frac{3}{5},
\end{align*}
and hence the value of human expertise is $\Delta = \sfrac{3}{5} - \sfrac{1}{2} = \sfrac{1}{10}$. 
\qed
\end{example}

\subsubsection{Cutting Plane Method} \label{sec:algorithm:cutting}
In applications where the reformulation techniques from \S\ref{sec:algorithm:reform} do not apply, we propose a two-level cutting plane method for solving the reduced robust optimization problem~\eqref{prob:robust_eta}. Let Assumption~\ref{ass:main} hold. The two-level cutting plane method is presented in Algorithm~\ref{alg:cutting}. 

In greater detail,  the two-level cutting plane method in Algorithm~\ref{alg:cutting} is motivated by applications where the policy space $\mathcal{X}$ is finite but has many elements and the constraint $\max_{\bx \in \mathcal{X}} f(\bx,\btheta) \le \eta$ in the reduced uncertainty set $\mathcal{U}_\eta$ does not have a compactly representable dual. Examples include traveling salesman problems or assortment optimization problems of the form studied in \S\ref{sec:assortment:example1} with large numbers of products.  Algorithm~\ref{alg:cutting}  addresses the fact that $\mathcal{U}_\eta$ cannot be represented compactly by using a cutting plane method to approximate the reduced uncertainty set; this corresponds to the loop within Step 2. Indeed, we observe that Step 2 in Algorithm~\ref{alg:cutting} begins with an $\bx^\star \in \mathcal{X}$ and outputs an optimal solution 
$ \btheta^\star \in \argmin_{\btheta \in \mathcal{U}_\eta} f(\bx^\star, \btheta)$. The outer cutting plane method solves $\max_{\bx \in \mathcal{X}} \min_{\btheta \in \mathcal{U}_\eta} f(\bx,\btheta)$ by solving $\max_{\bx \in \mathcal{X}} \min_{\btheta \in \tilde{\mathcal{U}}} f(\bx,\btheta)$ (Step 1) and adding parameters to $\tilde{\mathcal{U}}$ (Step 3) until convergence. The initial $\btheta \in \mathcal{U}_\eta$ in Step 0 can be found, for example, by running Step 2 once before the method begins with any initial $\bx \in \mathcal{X}$.\looseness=-1

The inner cutting plane method (Step 2 in Algorithm~\ref{alg:cutting}) maintains its set of cuts $\tilde{\mathcal{X}}$ across iterations of the outer cutting plane method (Steps 1-3 in Algorithm~\ref{alg:cutting}). This is intended to warm start the problem of solving $\min_{\btheta \in \mathcal{U}_\eta} f(\bx^\star,\btheta)$ in each iteration of the outer cutting plane method. For example, suppose one is considering a reduced robust optimization problem~$\max_{\bx \in  \mathcal{X} \subseteq \{0,1\}^n} \min_{\btheta \in \mathcal{U}_\eta} \btheta^\intercal \bx$ where the set $\mathcal{X}$ is the set of feasible tours in a traveling salesman problem and $\mathcal{U}$ is a compactly representable polytope. In this case, the inner cutting plane method (Step 2) requires iteratively solving the optimization problem $\max_{\bx \in \mathcal{X} \subseteq \{0,1\}^n} f(\bx,\btheta)$, which may be computationally expensive. By maintaining the set $\tilde{\mathcal{X}}$ of cuts on the reduced uncertainty set $\mathcal{U}_\eta$ across iterations, the problem $\max_{\bx \in \mathcal{X} \subseteq \{0,1\}^n} f(\bx,\btheta)$ may be solved fewer times in total across the iterations of the outer cutting plane method (Steps 1-3).

\begin{algorithm}[t] 
\begin{center}
\fbox{\begin{minipage}{\linewidth}{  \small
\singlespacing
\vspace{-1em}
\begin{center}
\textbf{\underline{Two-Level Cutting Plane Method}}
\end{center}
\begin{description}
\item[Step 0] Choose an initial $\btheta \in \mathcal{U}_\eta$ and let $\tilde{\mathcal{U}} \leftarrow \{\btheta\}$ and $\tilde{\mathcal{X}} \leftarrow \emptyset$. %
\item[Step 1] Find an optimal solution $(\bx^\star,t^\star)$ for the optimization problem 
\begin{align*}
\begin{aligned}
&\underset{\bx \in \mathcal{X}, t \in \R}{\text{maximize}} &&t \\
&\text{subject to}&& t \le f(\bx,\btheta) \quad \forall \btheta \in \tilde{\mathcal{U}} %
\end{aligned}
\end{align*}
The existence of an optimal solution follows from Assumption~\ref{ass:main} and the fact that $\tilde{\mathcal{U}}$ is a finite set.
\item[Step 2]  Do the following steps. 
\begin{description}
\item[Step 2a] Find an optimal solution $\btheta^\star$ for the optimization problem
\begin{align*}
\begin{aligned}
&\underset{\btheta \in \mathcal{U}}{\text{minimize}} &&f(\bx^\star,\btheta) \\
&\text{subject to}&& f(\bx,\btheta) \le \eta \quad \forall \bx \in \tilde{\mathcal{X}} %
\end{aligned}
\end{align*}
The existence of an optimal solution follows from Assumption~\ref{ass:main} and the fact that $\tilde{\mathcal{X}}$ is a finite set.\looseness=-1
\item[Step 2b] Find an optimal solution $\hat{\bx}$ for the optimization problem
\begin{align*}
\hat{v} = \max_{\bx \in \mathcal{X}} f(\bx,\btheta^\star)
\end{align*}
\item[Step 2c] If $\hat{v} > \eta$, then  let $\tilde{\mathcal{X}} \leftarrow \tilde{\mathcal{X}} \cup \{\hat{\bx} \}$ and return to Step 2a. If $\hat{v} \le \eta$, then proceed with $\btheta^\star$ to Step 3.\looseness=-1

\end{description}
\item[Step 3] If $t^\star > f(\bx^\star,\btheta^\star)$, then let $\tilde{\mathcal{U}} \leftarrow \tilde{\mathcal{U}} \cup \{\btheta^\star\}$ and return to Step 1. If $t^\star \le f(\bx^\star,\btheta^\star)$, then $\bx^\star$ is an optimal solution for the reduced robust optimization problem~\eqref{prob:robust_eta} 
\end{description}
}
\end{minipage}}
\end{center}
\caption{The cutting plane method from \S\ref{sec:algorithm:cutting} for the reduced robust optimization problem~\eqref{prob:robust_eta}.}\label{alg:cutting}
\end{algorithm}

\subsection{Alternative Approaches to Generating Policies } \label{appx:disappointment}

If none of the nominal curves~\eqref{line:wc_eta_set} in the menu of policies obtained by  the approach from \S\ref{sec:managerialexpertise:generating_policies}  are deemed satisfactory to a decision maker, then we suggest two alternative approaches for generating  policies. 

 The first alternative approach, which is based on a modeling technique from the literature on globalized robust optimization~\citep{ben2006extending}, is motivated by settings where the decision maker has a  concrete guess of an upper bound $\eta$ on the optimal value of the nominal problem~\eqref{prob:true} but wants to control the worst-case performance of the policy if their guess is incorrect. This approach obtains a policy by solving 
\begin{align}
\max_{\bx \in \mathcal{X}} \min_{\param \in \mathcal{U}} \left \{f(\bx,\param) + \lambda \max \left \{ \max_{\by \in \mathcal{X}} f(\by,\param) - \eta, 0 \right \} \right \}\label{prob:robust_augmented}
\end{align}
where  $\lambda \in [0,\infty)$ is a parameter selected by a decision maker that controls the disappointment from selecting an $\eta$ that is not an upper bound on the optimal value of \eqref{prob:true}.

To make sense of the above  alternative approach to obtaining a policy, let us make some observations. First, we observe that the optimal value of \eqref{prob:robust_augmented} is nondecreasing in $\lambda$. In the extreme case where $\lambda = 0$, we observe that \eqref{prob:robust_augmented} equals the standard robust optimization problem~\eqref{prob:robust}, and it is straightforward to see that \eqref{prob:robust_augmented}  simplifies to the reduced robust optimization problem~\eqref{prob:robust_eta} in the case where $\lambda \to \infty$. In the other cases of $\lambda$, the interpretation of $\lambda$ follows from the following Proposition~\ref{prop:globalized_interpretation}, which specifies the relationship between the reduced uncertainty set $\mathcal{U}_\eta$ and \eqref{prob:robust_augmented}. 
\begin{proposition} \label{prop:globalized_interpretation}
If $\bx^*$ is an optimal solution for \eqref{prob:robust_augmented} and $v(\eta,\lambda)$ is the optimal value of \eqref{prob:robust_augmented}, then\looseness=-1
\begin{alignat*}{2}
f(\bx^*,\param) &\ge v(\eta,\lambda) &&\forall \param \in \mathcal{U}_\eta,\\
f(\bx^*,\param) &\ge v(\eta,\lambda) - \lambda \left ( \max_{\bx \in \mathcal{X}} f(\bx,\param) - \eta \right ) &\quad &\forall \param \in \mathcal{U} \setminus \mathcal{U}_\eta.
\end{alignat*}
\end{proposition} 
\begin{proof}
If $\bx^*$ is an optimal solution for \eqref{prob:robust_augmented}, then for all $\param \in \mathcal{U}$, we have
\begin{align*}
f(\bx^*,\param) + \lambda \max \left \{ \max_{\by \in \mathcal{X}} f(\by,\param) - \eta, 0 \right \} \ge v(\eta,\lambda).
\end{align*}
Rearranging the above inequality and applying the definition of $\mathcal{U}_\eta$ yields the desired result.
\end{proof}
The above proposition shows that larger values of $\lambda$ imply that the performance of optimal solutions  for \eqref{prob:robust_augmented} is potentially less conservative, but the performance of those policies is degraded if $\eta$ is not an upper bound on \eqref{prob:true}. As such, we can  solve~\eqref{prob:robust_augmented} with varying choices of $\lambda \in [0,\infty)$ to obtain tradeoffs between performance guarantees and confidence about the accuracy of the estimate $\eta$. Similarly to the reduced robust optimization problem~\eqref{prob:robust_eta}, the inner problem of \eqref{prob:robust_augmented} is a convex optimization problem for a fixed policy under the assumptions from \S\ref{sec:mainresults}, as shown by the following proposition.
\begin{proposition}
If Assumptions~\ref{ass:main} and \ref{ass:convex} hold and $\lambda \in [0,\infty)$, then for each $\bx \in \mathcal{X}$, the inner problem $\min_{\param \in \mathcal{U}} \left \{f(\bx,\param) + \lambda \max \left \{ \max_{\by \in \mathcal{X}} f(\by,\param) - \eta, 0 \right \} \right \}$ is a convex optimization problem. 
\end{proposition}
\begin{proof}
The uncertainty set $\mathcal{U}$ is convex by Assumption~\ref{ass:convex}.  Given a fixed  policy $\bx \in \mathcal{X}$,  the objective function of the inner problem is the sum of a convex function $\btheta \mapsto f(\bx,\btheta)$ and the maximum of $0$ and a function of $\btheta \mapsto \lambda (\max_{\by \in \mathcal{X}} f(\by,\param) - \eta)$ that is convex since  $\lambda \in [0,\infty)$ and $\param \mapsto f(\by,\param)$ is convex for each $\by$. Therefore, the inner problem of \eqref{prob:robust_augmented} is a convex optimization problem. \looseness=-1
\end{proof}

Our second alternative approach is to add hard constraints into the reduced robust optimization problem~\eqref{prob:robust_eta} to restrict to policies that have acceptable worst-case performance guarantees under different possible upper bounds on the optimal value of \eqref{prob:true}. Given an $\eta_1 \ge \ubar{\eta}$ and an acceptable set $\mathcal{A} \equiv \{(\eta_2,\nu_2),\ldots,(\eta_m,\nu_m)\}\subseteq [\ubar{\eta},\infty) \times \R$ of pairs of upper bounds and worst-case performance guarantees, one can solve
\begin{equation} \label{prob:ilikeless}
\begin{aligned}
&\underset{\bx \in \mathcal{X}}{\text{maximize}}&& \min_{\param \in \mathcal{U}_{\eta_1}} f(\bx,\param) \\
&\text{subject to}&& \min_{\param \in \mathcal{U}_{\eta_i}} f(\bx,\param) \ge \nu_i && \forall i \in \{2,\ldots,m\}
\end{aligned}
\end{equation}
The problem~\eqref{prob:ilikeless} gives the option to have fine-grained control over the performance of a policy under different bounds $\eta_1,\ldots,\eta_m$. Similarly to the first alternative approach~\eqref{prob:robust_augmented}, a menu of policies can be obtained by solving~\eqref{prob:ilikeless} with varying choices of acceptable sets.

\section{Conclusion and Open Questions} \label{sec:conclusion}

In this work, we studied optimization applications with unknown parameters where the decision maker believes that the optimal value of the nominal problem is unlikely to be large.  We proposed nominal curves for translating such beliefs into performance guarantees while retaining guarantees when the belief is incorrect. We introduced a quantity called the \emph{value of human expertise} that captures the maximum improvement in performance guarantees from incorporating such beliefs, and our main result showed under mild assumptions that this quantity is equal to the gap between the optimal values of a min-max and max-min problem. We showed that this gap can be substantial in applications such as assortment optimization and shortest path problems, and  that improved performance guarantees can be obtained even under relatively loose beliefs about the nominal problem. We also proposed several approaches to generating menus of policies, as well as computational methods for finding policies that are pointwise optimal. 

There are many important directions for future work, including empirical testing, case studies, and connections to fields such as decision theory. %
One  question that could be interesting to investigate is whether the \emph{optimizer's curse} phenomenon \citep{smith2006optimizer,van2021data,gupta2024debiasing,xu2025winner,bastani2025beating} can be exploited to construct statistical upper bounds on the optimal value of the nominal problem that lead to improved performance guarantees. Another direction would be to study optimization with human expertise in specific application classes, such as mechanism design and stochastic programming with chance constraints. It would also be interesting to study the use of nominal curves in human-AI collaboration, such as in settings where agents are deployed to solve high-stakes operations management  problems. \looseness=-1

\bibliographystyle{plainnat}
\bibliography{bib}
\clearpage
\appendix


\section{Review of Topology} \label{app:topology}
Our analysis in \S\S\ref{sec:managerialexpertise} and \ref{sec:mainresults} utilizes basic facts about topology, which we review here. 
\subsection{Compact sets and semicontinuity} \label{app:topology_basics}

Recall from \S\ref{sec:managerialexpertise:setting}  that $\mathcal{X}$ and $\mathcal{U}$ are assumed throughout the paper to be compact nonempty sets (this is stated formally as Assumption~\ref{ass:main} in \S\ref{sec:managerialexpertise:computation}). It follows from the definition of lower and upper semicontinuity \cite[Definition 2.8]{rudin1987real}   that if $\psi: \mathcal{U} \to \R$ is lower semicontinuous and $\phi: \mathcal{X} \to \R$ is upper semicontinuous, then the sets\looseness=-1
\begin{align*}
\{\param \in \mathcal{U}: \psi(\param) \le \eta\} \text{ and }\{\bx \in \mathcal{X}: \phi(\bx) \ge \alpha\}
\end{align*} are closed for all $\eta, \alpha \in \R$. Because the above sets are closed subsets of compact sets $\mathcal{U}$ and $\mathcal{X}$, we also have that the above sets are compact sets  for all $\eta,\alpha \in \R$ \cite[Theorem 2.4]{rudin1987real}. 

We will utilize several additional properties of lower and upper semicontinuous functions. First, we will use the fact that if $(\psi_i)$ are lower semicontinuous functions, then $\max_i \psi_i$ is a lower semicontinuous function (similarly, $\min_i \phi_i$ is upper semicontinuous if $(\phi_i)$ are upper semicontinuous functions). This, for example, implies that the reduced uncertainty set from \S\ref{sec:managerialexpertise}, defined as 
\begin{align*}
\mathcal{U}_\eta &\triangleq \left \{\param \in \mathcal{U}: \max_{\bx \in \mathcal{X}} f(\bx,\param) \le \eta \right \},
\end{align*} 
is a compact set for all $\eta \in \R$. Second, we will utilize semicontinuity in the context of the Weierstrass extreme value theorem, which says that the maximum of a bounded upper semicontinuous function over a compact set is attained, and similarly for the minimum of a bounded lower semicontinuous function over a compact set. This implies that the inner and outer problems in \eqref{prob:robust_eta}, \eqref{prob:robust} and \eqref{prob:robust_exchange}  attain their optimums.

Finally,  for a general topological space, a set is compact if every open cover of the set has a finite subcover \cite[Definition 2.3]{rudin1987real}. From this definition of compact sets, we obtain the following proposition, which we will use in Appendix~\ref{app:topology_application}.

\begin{proposition} \label{prop:compactintersection}
Let $\{\mathcal{X}_i\}_{i \in I}$ be a collection of closed sets that satisfy $\mathcal{X}_i \subseteq \mathcal{X}$ for all $i \in I$. If $\cap_{i \in I} \mathcal{X}_i = \emptyset$, then there exists a finite collection $\{i_1,\ldots,i_K\} \subseteq I$ that satisfies $\mathcal{X}_{i_1} \cap \cdots \cap \mathcal{X}_{i_K} = \emptyset$.
\end{proposition}
\begin{proof}Define  $\mathcal{Y}_i \triangleq \mathcal{X} \setminus \mathcal{X}_i$ for all $i \in I$. It follows from the fact that each $\mathcal{X}_i$  is  a closed set that each $\mathcal{Y}_i$ is an open set, and it follows from the fact that $\cap_{i \in I} \mathcal{X}_i = \emptyset$ that  $\bigcup_{i \in I} \mathcal{Y}_i = \mathcal{X}$. Because $(\mathcal{Y}_i: i \in I)$ is a collection of open sets that forms a cover of $\mathcal{X}$, and because $\mathcal{X}$ is a compact set, there must exist a finite collection  $\{i_1,\ldots,i_K\}$ that satisfies  $\mathcal{Y}_{i_1} \cup \cdots \cup \mathcal{Y}_{i_K} = \mathcal{X}$, which implies that $\mathcal{X}_{i_1} \cap \cdots \cap \mathcal{X}_{i_K} = \emptyset$.
\end{proof}

\subsection{Omitted details from the proof of Theorem~\ref{thm:fundamental}} \label{app:topology_application}
Let $\eta = \min_{\param \in \mathcal{U}} \max_{\bx \in \mathcal{X}} f(\bx,\param)$ and $\mathcal{U}^* = \argmin_{\param \in \mathcal{U}} \max_{\bx \in \mathcal{X}} f(\bx,\param) = \mathcal{U}_\eta$. For each $\param \in \mathcal{U}^*$ define
\begin{align}
    \mathcal{X}^\param &\triangleq \left \{ \bx \in \mathcal{X}: f(\bx,\param) \ge \eta \right \}\label{line:X_c_oneway}
\end{align}
It follows from the fact that $\bx \mapsto f(\bx,\param)$ is upper semicontinuous and from Appendix~\ref{app:topology} that $\mathcal{X}^\param$ is a closed set for all $\param \in \mathcal{U}^*$. Furthermore, it follows from the definition of $\eta$ that for all $\param \in \mathcal{U}^*$ and all $\bx \in \mathcal{X}$, we have
\begin{align}
f(\bx,\param) \le \max_{\hat{\bx} \in \mathcal{X}} f(\hat{\bx},\param) = \eta,\label{line:X_c_otherway}
\end{align} 
where the inequality follows from algebra and the equality follows from the fact that $\param \in \mathcal{U}^*$. Combining lines~\eqref{line:X_c_oneway} and \eqref{line:X_c_otherway}, we have shown for all $\param \in \mathcal{U}^*$ that
\begin{align*}
    \mathcal{X}^\param &= \left \{ \bx \in \mathcal{X}: f(\bx,\param)= \eta \right \}.
\end{align*}

It is supposed at the beginning of Case 2 in the proof of Theorem~\ref{thm:fundamental} from \S\ref{sec:maintheorem} that there does not exist an $\hat{\bx} \in \mathcal{X}$ that satisfies  $f(\hat{\bx},\param) = \eta$ for all $\param \in \mathcal{U}^*$. Using the above notation, this is equivalent to supposing that\looseness=-1
 \begin{align*}
    \bigcap_{\param \in \mathcal{U}^*} \mathcal{X}^\param = \emptyset.
 \end{align*}
 We observe that $\{\mathcal{X}^\param\}_{\param \in \mathcal{U}^*}$ is a collection of closed sets that satisfy $\mathcal{X}^\param \subseteq \mathcal{X}$ for all $\param \in \mathcal{U}^*$. Therefore, it follows from Proposition~\ref{prop:compactintersection} from Appendix~\ref{app:topology} that there exists a finite collection $\{\param_1,\ldots,\param_K \} \subseteq \mathcal{U}^*$ that satisfies
    \begin{align*}
    \bigcap_{k=1}^K \mathcal{X}^{\param_k}  = \emptyset,
    \end{align*}
    which completes the omitted details from the beginning of Case 2 in the proof of Theorem~\ref{thm:fundamental}.


\section{Proof of Proposition~\ref{prop:shape}} \label{appx:misc_eta_nominal}

Let $\ubar{\eta} = \min_{\btheta \in \mathcal{U}} \max_{\by \in \mathcal{X}} f(\by,\btheta)$.  The proof that $v_\bx(\eta)$ is nonincreasing follows from the fact that $\mathcal{U}_{\eta_1} \subseteq \mathcal{U}_{\eta_2}$ for all $\eta_1 \le \eta_2$.

 To prove that $v_\bx(\cdot)$ is convex, consider any $\eta_1 < \eta_2$ that satisfy $\eta_1, \eta_2 \in [\ubar{\eta},\infty)$ and $\lambda \in (0,1)$.  It follows from  Assumption~\ref{ass:convex} and Proposition~\ref{prop:eta} that the uncertainty set $\mathcal{U}_\eta$ is convex and nonempty for all $\eta \in [\ubar{\eta},\infty)$. Consider any optimal solutions $\btheta_i \in \argmin_{\btheta \in \mathcal{U}_{\eta_i}} f(\bx,\btheta)$ for $i \in \{1,2\}$, and define $\hat{\eta} =  \lambda \eta_1 + (1-\lambda) \eta_2$ and $\hat{\btheta} \coloneqq \lambda \btheta_1 + (1-\lambda) \btheta_2$. It follows from Assumption~\ref{ass:convex} that $\hat{\btheta} \in \mathcal{U}$. Moreover, we have 
\begin{align*}
\max_{\by \in \mathcal{X}} f(\by,\hat{\btheta}) \le \lambda \max_{\by \in \mathcal{X}} f(\by,\btheta_1) +  (1-\lambda) \max_{\by \in \mathcal{X}} f(\by,\btheta_2)  \le \lambda \eta_1 + (1-\lambda) \eta_2 = \hat{\eta}
\end{align*}
where the first inequality follows from Assumption~\ref{ass:convex} and the second follows from the fact that $\btheta_i \in \mathcal{U}_{\eta_i}$ for $i \in \{1,2\}$. We thus conclude that $\hat{\btheta} \in \mathcal{U}_{\hat{\eta}}$. Therefore, 
\begin{align*}
v_\bx(\hat{\eta}) \le f(\bx, \hat{\btheta}) \le \lambda f(\bx,\btheta_1) + (1-\lambda) f(\bx,\btheta_2)  \le \lambda v_\bx(\eta_1) + (1-\lambda) v_\bx(\eta_2)
\end{align*}
where the inequality follows from the fact that $\hat{\btheta} \in \mathcal{U}_{\hat{\eta}}$, the second inequality follows from Assumption~\ref{ass:convex}, and the third inequality follows from the definition of $\btheta_1,\btheta_2$. We have thus shown that $v_\bx(\cdot)$ is convex. 

We conclude by showing that $v_\bx(\cdot)$ is continuous on $[\ubar{\eta},\infty)$. It follows from the fact that $v_\bx(\cdot)$  is convex that $v_\bx(\cdot)$ is continuous on $(\ubar{\eta},\infty)$. To show that the continuity extends to $\eta = \ubar{\eta}$,  choose any sequence $\eta_k \downarrow \ubar{\eta}$ and let $\btheta_k \in \argmin_{\btheta \in \mathcal{U}_{\eta_k}} f(\bx,\btheta)$ for all $k$. It follows from compactness of $\mathcal{U}$ (Assumption~\ref{ass:main}) that the net $(\btheta_k)$ has a convergent subnet $(\btheta_{k_\alpha})_{\alpha \in A}$  with $\btheta_{k_\alpha} \to \hat{\btheta}$ for some $\hat{\btheta} \in \mathcal{U}$. It thus follows from the lower semicontinuity of  $\max_{\by \in \mathcal{X}} f(\by,\cdot)$ (Assumption~\ref{ass:main}) that %
\begin{align}
\max_{\by \in \mathcal{X}} f(\by,\hat{\btheta})  &\le \liminf_{\alpha} \max_{\by \in \mathcal{X}} f(\by,\btheta_{k_\alpha}) \le \lim_\alpha \eta_{k_\alpha} = \ubar{\eta}\label{line:liminf_1}
\end{align}
where the first inequality follows from lower semicontinuity, the second inequality follows from the fact that $\btheta_{k_\alpha} \in \mathcal{U}_{\eta_{k_\alpha}}$ for all $\alpha$, and the equality holds because every subnet of $(\eta_k)$ converges to $\ubar{\eta}$. 
Moreover, we have
\begin{align}
 f(\bx,\hat{\btheta}) & \le \liminf_{\alpha} f(\bx,\btheta_{k_\alpha}) = \lim_\alpha v_\bx(\eta_{k_\alpha})\label{line:liminf_2}
\end{align}
 where the first inequality follows from the fact that  $f(\bx,\cdot)$ is lower semicontinuous (Assumption~\ref{ass:main}), and the equality follows from the definition of $\btheta_{k_\alpha}$ and the fact that every subnet of $(v_\bx(\eta_k))$ converges to the same limit (which exists because $v_\bx(\cdot)$ is non-increasing, is bounded because of Assumption~\ref{ass:main}, and $\eta_k \downarrow \eta$). 
We thus have
\begin{align}
v_\bx(\ubar{\eta}) \le f(\bx,\hat{\btheta}) \le  \lim_{\alpha} v_\bx(\eta_{k_\alpha}) \label{line:liminf_3}
\end{align}
where the first inequality follows from the fact that $\hat{\btheta} \in \mathcal{U}_{\ubar{\eta}}$ as shown in \eqref{line:liminf_1} and the second inequality is \eqref{line:liminf_2}. Since we established previously that $v_\bx(\cdot)$ is nonincreasing, it follows from the fact that $f(\bx,\cdot)$ is bounded (Assumption~\ref{ass:main}) that $v_\bx(\ubar{\eta}) \ge \lim_{k \to \infty} v_\bx(\eta_{k}) $. Combining that inequality with \eqref{line:liminf_3}, we conclude that $v_\bx(\ubar{\eta}) = \lim_{k \to \infty} v_\bx(\eta_k) $, which completes the proof that $v_\bx(\cdot)$ is continuous on $[\ubar{\eta},\infty)$. 


\section{Omitted Proofs From \S\ref{sec:structure:budget}} \label{app:proofs_structure}
\begin{proof}[Proof of Theorem~\ref{thm:extreme}]
To show the first direction, suppose there exists a vector $\tilde{\bz}$ that is an optimal solution of $\max_{\bz \in \mathcal{Z}} \sum_{j \in \tilde{I}} (\hat{c}_j + d_j z_j)$ and satisfies $\text{supp}(\tilde{\bz}) \cap I^* \neq \emptyset$ for every optimal solution $I^*$ of the non-robust problem~\eqref{prob:nonrobust}. In that case, we observe for every optimal solution $I^*$ of the non-robust problem~\eqref{prob:nonrobust} that 
\begin{align*}
\sum_{j \in I^*} (\hat{c}_j + d_j \tilde{z}_j) \ge  \sum_{j \in I^*} \hat{c}_j  + \sum_{j \in  \text{supp}(\tilde{\bz}) \cap I^*} d_j \tilde{z}_j > \sum_{j \in I^*} \hat{c}_j  = \min_{I \in \mathcal{I}}  \sum_{j \in I} \hat{c}_j,
\end{align*}
Indeed, the first inequality follows from algebra. The strict inequality follows from the fact that $\text{supp}(\tilde{\bz}) \cap I^* \neq \emptyset$, the fact that $\bd > \bzero$, and the definition of $\text{supp}(\tilde{\bz})$. The equality follows from the definition of $I^*$. Moreover, we observe for every non-optimal solution $I'$ of the non-robust problem~\eqref{prob:nonrobust} that
\begin{align*}
\sum_{j \in I'} (\hat{c}_j + d_j \tilde{z}_j) \ge \sum_{j \in I'} \hat{c}_j > \min_{I \in \mathcal{I}} \sum_{j \in I} \hat{c}_j
\end{align*}
where the inequality follows from the fact that $\bd > \bzero$ and the strict inequality follows from the fact that $I'$ is not an optimal solution for \eqref{prob:nonrobust}. We have thus  shown in all cases that
\begin{align*}
\min_{I \in \mathcal{I}} \sum_{j \in I} (\hat{c}_j + d_j \tilde{z}_j)  > \min_{I \in \mathcal{I}} \sum_{j \in I} \hat{c}_j.
\end{align*}
 We conclude that  $\tilde{\bz} \in \mathcal{Z}_\eta$ for all $\eta  \in \mathbb{H} \cap (\min_{I \in \mathcal{I}} \sum_{j \in I} \hat{c}_j,\min_{I \in \mathcal{I}}\sum_{j \in I} (\hat{c}_j + d_j \tilde{z}_j) ] \neq \emptyset$, and so it follows from the fact that $\tilde{\bz}$ is an  optimal solution of $\max_{\bz \in \mathcal{Z}} \sum_{j \in \tilde{I}}(\hat{c}_j + d_j z_j)$ that
\begin{align*}
\max_{\bz \in \mathcal{Z}_\eta}   \sum_{j \in \tilde{I}}  \left( \hat{c}_j + d_j z_j  \right) = \max_{\bz \in \mathcal{Z}}   \sum_{j \in \tilde{I}}  \left( \hat{c}_j + d_j z_j  \right)
\end{align*}

To show the other direction, consider any arbitrary vector $\tilde{\bz}$ that is an optimal solution of $\max_{\bz \in \mathcal{Z}} \sum_{j \in \tilde{I}} (\hat{c}_j + d_j z_j)$, and suppose that there exists  an optimal solution $I^*$ of the non-robust problem~\eqref{prob:nonrobust} that satisfies $\text{supp}(\tilde{\bz}) \cap I^* = \emptyset$. We observe for each $\eta \in \mathbb{H}$ that  
\begin{align*}
\mathcal{Z}_\eta \subseteq   \left \{ \bz \in \mathcal{Z}:  \sum_{j \in I^*} (\hat{c}_{j} + d_j z_j)  \ge \eta \right \} \subseteq  \left \{ \bz \in \mathcal{Z}: \sum_{j \in I^*} \left( \hat{c}_j + d_j z_j \right) > \sum_{j \in I^*} \hat{c}_j \right \} =  \left \{ \bz \in \mathcal{Z}: \sum_{j \in I^*} d_j z_j > 0 \right \}, 
\end{align*}
where the first inclusion follows from the definition of $\mathcal{Z}_\eta$,  the second inclusion follows from the fact that $\eta \in \mathbb{H}$ and that $I^*$ is an optimal solution for \eqref{prob:nonrobust}, and the equality follows from the fact that  $\bd > \bzero$ and $\mathcal{Z} \subseteq [0,1]^n$.  Moreover, it follows from the fact that $\text{supp}(\tilde{\bz}) \cap I^* = \emptyset$ that $\sum_{j \in I^*} d_j z_j = 0$, which implies that $\tilde{\bz} \notin \mathcal{Z}_\eta$. 
Since the optimal solution $\tilde{\bz}$ of $\max_{\bz \in \mathcal{Z}} \sum_{j \in \tilde{I}} (\hat{c}_j + d_j z_j)$ was chosen arbitrarily, we conclude that if for every optimal solution $\tilde{\bz}$ of $\max_{\bz \in \mathcal{Z}} \sum_{j \in \tilde{I}} (\hat{c}_j + d_j z_j)$ there exists  an optimal solution $I^*$ of the non-robust problem~\eqref{prob:nonrobust} that satisfies $\text{supp}(\tilde{\bz}) \cap I^* = \emptyset$, then
\begin{align*}
\max_{\bz \in \mathcal{Z}_\eta}   \sum_{j \in \tilde{I}}  \left( \hat{c}_j + d_j z_j  \right) < \max_{\bz \in \mathcal{Z}}   \sum_{j \in \tilde{I}}  \left( \hat{c}_j + d_j z_j  \right)
\end{align*}
\end{proof}

\begin{proof}[Proof of Corollary~\ref{cor:extreme}]
Let  $I^{\text{RO}}$ denote an optimal solution of the robust optimization problem~\eqref{prob:lp_robust}.

Suppose that $I^{\text{RO}}$ satisfies condition \ref{suffcond:a}, meaning that $|I^{\text{RO}}| \ge \Gamma$ and there exists an optimal solution $I^*$ for the non-robust problem~\eqref{prob:nonrobust} that satisfies $I^{\text{RO}} \cap I^* = \emptyset$. Consider any arbitrary $\tilde{\bz} \in \argmax_{\bz \in \mathcal{Z}} \sum_{j \in I^{\text{RO}}} (\hat{c}_j + d_j z_j)$. It follows from the fact that $\bd > \bzero$, the fact that $\Gamma$ is an integer, the fact that $|I^{\text{RO}}| \ge \Gamma$, and the definition of the budget uncertainty set $\mathcal{Z} = \{\bz \in [0,1]^n: \sum_{j =1}^n z_j \le \Gamma\}$ that $\tilde{\bz}$ can satisfy $\tilde{z}_j > 0$ only if $j \in I^{\text{RO}}$. This implies that $\text{supp}(\tilde{\bz}) \subseteq I^{\text{RO}}$, and so it follows from the fact that $I^{\text{RO}} \cap I^* = \emptyset$ that $\text{supp}(\tilde{\bz}) \cap I^* = \emptyset$. Since $\tilde{\bz}$ was chosen arbitrarily, we conclude for all $\tilde{\bz} \in \argmax_{\bz \in \mathcal{Z}} \sum_{j \in I^{\text{RO}}} (\hat{c}_j + d_j z_j)$ that $\text{supp}(\tilde{\bz}) \cap I^* = \emptyset$, and so Theorem~\ref{thm:extreme} implies that $I^{\text{RO}}$ satisfies \eqref{line:unnec}. 

Alternatively, suppose that $I^{\text{RO}}$ satisfies condition \ref{suffcond:b}. Let $\tilde{\bz}$ denote the unique optimal solution for $\max_{\bz \in \mathcal{Z}} \sum_{j \in I^{\text{RO}}} (\hat{c}_j + d_j z_j)$, and let  $I^*$ denote the optimal solution of the non-robust problem~\eqref{prob:nonrobust} that satisfies $\text{supp}(\tilde{\bz}) \cap I^* = \emptyset$. Since $\tilde{\bz}$ is the unique optimal solution, it follows immediately from Theorem~\ref{thm:extreme} that \eqref{line:unnec} is satisfied. 

In summary, we have shown that if condition~\ref{suffcond:a} or \ref{suffcond:b} is satisfied, then $I^{\text{RO}}$ satisfies \eqref{line:unnec}. We thus conclude for each $\eta \in \mathbb{H}$ that 
\begin{align*}
\min_{I \in \mathcal{I}} \max_{\bz \in \mathcal{Z}}   \sum_{j \in I}  \left( \hat{c}_j + d_j z_j  \right) &= \max_{\bz \in \mathcal{Z}}   \sum_{j \in I^{\text{RO}}}  \left( \hat{c}_j + d_j z_j  \right) > \max_{\bz \in \mathcal{Z}_\eta}   \sum_{j \in I^{\text{RO}}}  \left( \hat{c}_j + d_j z_j  \right) \ge \min_{I \in \mathcal{I}} \max_{\bz \in \mathcal{Z}_\eta}   \sum_{j \in I}  \left( \hat{c}_j + d_j z_j  \right)
\end{align*}
where the equality holds because $I^{\text{RO}}$ is an optimal solution for the robust optimization problem~\eqref{prob:lp_robust}, the first inequality follows from the fact that $I^{\text{RO}}$ satisfies \eqref{line:unnec} and the fact that $\eta \in \mathbb{H}$, and the second inequality holds because $I^{\text{RO}}$ is a feasible but possibly suboptimal solution for the reduced robust optimization problem~\eqref{prob:lp_eta}.

\end{proof}

\begin{proof}[Proof of Proposition~\ref{prop:broader}.]
It follows from Lemma~\ref{lem:lp_relationships} that if  $\eta = \min_{I \in \mathcal{I}} \sum_{j \in I} \hat{c}_j  $, then $\mathcal{Z} = \mathcal{Z}_\eta$. Moreover, if either of the sufficient conditions~\ref{suffcond:a} or \ref{suffcond:b} from Corollary~\ref{cor:extreme} is satisfied,  then it follows from Theorem~\ref{thm:extreme} that \looseness=-1
\begin{align*}
\max_{\bz \in \mathcal{Z}_\eta}   \sum_{j \in I^{\text{RO}}}  \left( \hat{c}_j + d_j z_j  \right) < \max_{\bz \in \mathcal{Z}}   \sum_{j \in I^{\text{RO}}}  \left( \hat{c}_j + d_j z_j  \right)  \quad \forall \eta \in \mathbb{H}
\end{align*}
Combining the above reasoning with the definitions of $\mathcal{Z}_\eta$ from \eqref{line:comb_unc_eta} and $\mathbb{H}$ from \eqref{line:H}, we conclude that\looseness=-1 
\begin{align*}
\argmax_{\bz \in \mathcal{Z}}  \sum_{j \in I^{\text{RO}}}  \left( \hat{c}_j + d_j z_j  \right) &= \argmax_{\bz \in \mathcal{Z} \setminus \mathcal{Z}_\eta}  \sum_{j \in I^{\text{RO}}}  \left( \hat{c}_j + d_j z_j  \right) \quad \forall \eta \in \mathbb{H}\\
&= \bigcap_{\eta \in \mathbb{H}} \argmax_{\bz \in \mathcal{Z} \setminus \mathcal{Z}_\eta}  \sum_{j \in I^{\text{RO}}}  \left( \hat{c}_j + d_j z_j  \right) \\
&\subseteq \left \{ \bz \in \mathcal{Z}: \min_{I \in \mathcal{I}} \sum_{j \in I} (\hat{c}_{j} + d_j z_j) =   \min_{I \in \mathcal{I}} \sum_{j \in I} \hat{c}_j  \right \}
\end{align*}
which completes the proof of Proposition~\ref{prop:broader}.
\end{proof}

\begin{proof}[Proof of Lemma~\ref{lem:reduced_comb}]
We  observe that the reduced uncertainty set can be rewritten as 
\begin{align*}
\mathcal{Z}_\eta &= \left \{  \bz \in \mathcal{Z}:   \sum_{j \in I}  \left( \hat{c}_j + d_j z_j  \right)  \ge \eta \; \forall I \in \mathcal{I} \right  \} = \left \{  \bz \in \mathcal{Z}:   \sum_{j \in I}  d_j z_j   \ge \eta -  \sum_{j \in I}  \hat{c}_j  \; \forall I \in \mathcal{I} \right  \} \\
&= \left \{  \bz \in \mathcal{Z}:   \sum_{j \in I}  d_j z_j   \ge  \eta -  \sum_{j \in I}  \hat{c}_j   \; \forall I \in \mathcal{I}_\eta \right  \}. 
\end{align*}
The first equality follows from the definition of $\mathcal{Z}_\eta$. The second equality follows from algebra. The third equality follows from the fact that $\mathcal{Z} \subseteq [0,1]^n$ and $\bd > \bzero$, and the fact that $\eta - \sum_{j \in I}  \hat{c}_j > 0$ if and only if $I \in \mathcal{I}_\eta$.\looseness=-1
\end{proof}

\begin{proof}[Proof of Proposition~\ref{prop:hitting}]
Our proof will make use of the following claim. 
\begin{claim} \label{claim:company}
If  $\eta \in \mathbb{H}$ and $\textsc{HittingSetNum}(\mathcal{I}_\eta) > \Gamma$, then $\mathcal{Z}_\eta \cap \{0,1\}^n = \emptyset$. 
\end{claim}
\begin{proof}[Proof of Claim~\ref{claim:company}]
Let $\eta \in \mathbb{H}$. We observe that
\begin{align}
\mathcal{Z}_\eta  \cap \{0,1\}^n =  \left \{  \bz \in \mathcal{Z} \cap \{0,1\}^n:   \sum_{j \in I}  d_j z_j   \ge \eta -  \sum_{j \in I}  \hat{c}_j  \forall I \in \mathcal{I}_\eta \right  \} \subseteq \left \{  \bz \in \{0,1\}^n: \begin{aligned}
&  \sum_{j \in I}   z_j   >0  \forall I \in \mathcal{I}_\eta\\
& \sum_{j = 1}^n z_j \le \Gamma
\end{aligned} \right  \} \label{line:bs_line}
\end{align}
where the equality follows from Lemma~\ref{lem:reduced_comb} and the set inclusion follows from the definition of $\mathcal{I}_\eta$, the fact that $\bd > \bzero$, and the definition of the uncertainty set $\mathcal{Z}$ from line~\eqref{line:budget_uncertaintyset}. If  $\textsc{HittingSetNum}(\mathcal{I}_\eta) > \Gamma$, then it follows from line~\eqref{line:hittingset} that the right-most set in line~\eqref{line:bs_line} is empty. This completes the proof of Claim~\ref{claim:company}. 
\end{proof}
Equipped with the above intermediary claim, our proof of Proposition~\ref{prop:hitting} is as follows. Let $\eta \in \mathbb{H}$  satisfy $\textsc{HittingSetNum}(\mathcal{I}_\eta) > \Gamma$, and 
suppose that $I^{\text{RO}}$ is an optimal solution of the robust optimization problem~\eqref{prob:lp_robust} for which $\max_{\bz \in \mathcal{Z}} \sum_{j \in I^{\text{RO}}} (\hat{c}_j + d_j z_j)$ has a unique solution. Let $\tilde{\bz} \in \argmax_{\bz \in \mathcal{Z}} \sum_{j \in I^{\text{RO}}} (\hat{c}_j + d_j z_j)$ denote the unique worst-case realization. Because it is unique, and because $\Gamma$ is an integer, $\tilde{\bz}$ must be  an extreme point of $\mathcal{Z}$ and thus must satisfy $\tilde{\bz} \in \{0,1\}^n$. %
It thus follows from Claim~\ref{claim:company} that $\tilde{\bz} \notin \mathcal{Z}_\eta$, which implies that 
\begin{align*}
\max_{\bz \in \mathcal{Z}_\eta}   \sum_{j \in I^{\text{RO}}}  \left( \hat{c}_j + d_j z_j  \right) <  \max_{\bz \in \mathcal{Z}}   \sum_{j \in I^{\text{RO}}}  \left( \hat{c}_j + d_j z_j  \right) = \min_{I \in \mathcal{I}} \max_{\bz \in \mathcal{Z}}   \sum_{j \in I}  \left( \hat{c}_j + d_j z_j  \right). 
\end{align*}
This completes the proof of Proposition~\ref{prop:hitting}. 
\end{proof}


\section{Extended Numerical Results from \S\ref{sec:assortment:results}} \label{appx:assortment_extra}
Figure~\ref{fig:assortment_big} is an expanded version of Figure~\ref{fig:assortment_small} from  \S\ref{sec:assortment:results} to show the nominal curves for a larger collection of assortments. Specifically, Figure~\ref{fig:assortment_big} shows the nominal curves for all assortments $S \in \mathscr{S}$ that satisfy $4 \in S$. It is sufficient to consider only the  assortments that satisfy $4 \in S$ because adding the most expensive product to an assortment will not decrease its performance. Figure~\ref{fig:assortment_big} shows that the assortment $\{0,2,3,4\}$ is optimal for the reduced robust optimization problem for all $\$23.875\le \eta \le \tilde{\eta}$ where $\tilde{\eta} \approx \$33.778$. The y-axis in Figure~\ref{fig:assortment_big} is clipped at \$15. \looseness=-1
\afterpage{%
\null
\vfill
       \begin{figure}[H]
        \centering
          \vspace*{\fill}
        \includegraphics[width=0.95\textwidth]{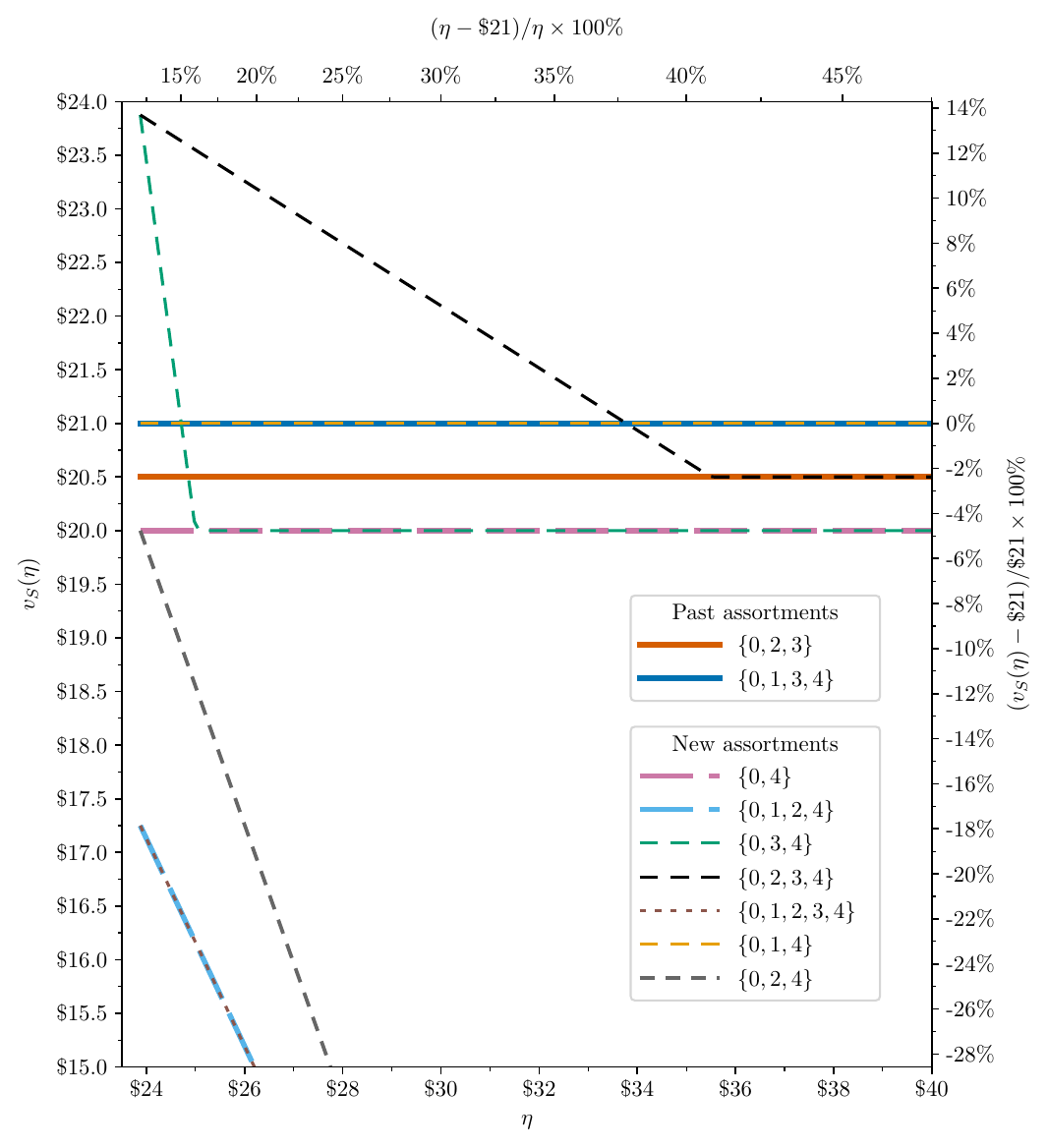}
        \caption{Additional nominal curves from numerical example in \S\ref{sec:assortment:results}.\looseness=-1}
         \label{fig:assortment_big}
  \vspace*{\fill}
    \end{figure}
    \null
\vfill
    \clearpage
    }
    \clearpage


\section{Omitted Proofs from \S\ref{sec:algorithm:reform}} \label{appx:messy_reforms}
 In this appendix, we consider the setting described in \S\ref{sec:algorithm:reform} where $\mathcal{X} = \{\bx \in \{0,1\}^n: \bba \bx \le \bb \} \neq \emptyset$  satisfies $\conv(\mathcal{X}) = \{ \bx \in \R^n: \bba \bx \le \bb \}$,  the uncertainty set  is a bounded polyhedron $\mathcal{U} = \{\btheta \in \R^p: \bbd \btheta \ge \bg \} \neq \emptyset$, and the dimensions are $\bba \in \R^{m \times n}$, $\bb \in \R^m$, $\bbd \in \R^{q \times p}$, $\bg \in \R^q$,  $\bbu, \bbl \in \R^{n \times p}$, and $\mu,\lambda \in \R$. Our main reformulation steps are contained in the following Lemma~\ref{lem:reformulation}. After presenting and proving Lemma~\ref{lem:reformulation}, we present the proofs of Propositions~\ref{prop:reform:simple} and \ref{prop:reform:general}. 
 \begin{lemma} \label{lem:reformulation}
 If $\mathcal{U}_\eta$ is nonempty and $\bx^\intercal \bbl {\btheta} + \lambda > 0$ for all $\bx \in \mathcal{X}$ and $\btheta \in \mathcal{U}$, then
 \begin{align*}
\max_{\bx \in \mathcal{X}}\min_{\btheta \in \mathcal{U}_\eta}  \frac{\bx^\intercal \bbu {\btheta} + \mu}{\bx^\intercal \bbl {\btheta} + \lambda} &=  \left[ \begin{aligned}
&\underset{ \bx \in \{0,1\}^n, t \in \R, \bpsi \in \R^q, \by \in \R^n, \sigma \in \R, \bz \in \R^n}{\textnormal{maximize}} &&    t\\
&\textnormal{subject to}&&
t \lambda - \mu \le \bg^\intercal \bpsi - (\eta \lambda - \mu) \sigma\\
    &&&  \bbd^\intercal \bpsi - \left( \bbu - \eta \bbl \right)^\intercal \by
        = \bbu^\intercal \bx  -  \bbl^\intercal \bz  \\
        &&& \bz = t \bx\\
        &&& \bba \bx \le \bb\\
 &&&     \bba \by \le \bb \sigma\\
 &&&  \bpsi \ge \bzero,\sigma \ge 0
    \end{aligned}\right] 
\end{align*}
 \end{lemma}
\begin{proof}Our proof follows from applying strong duality once to reformulate the reduced uncertainty set $\mathcal{U}_\eta$ as a polyhedron with a polynomial number of constraints, followed by dualizing the inner problem of the reduced robust optimization problem~\eqref{prob:fraction_eta}. Indeed, we  observe for each  $\btheta \in \mathcal{U}$ that
\begin{align*}
\max_{\bx \in \mathcal{X}} \frac{\bx^\intercal \bbu {\btheta} + \mu}{\bx^\intercal \bbl {\btheta} + \lambda} \le \eta &\iff \max_{\bx \in \mathcal{X}}  \left \{  \bx^\intercal \left( \bbu  - \eta \bbl \right) {\btheta}  \right \} \le \eta \lambda - \mu\\
&\iff \max_{\bx \in \conv( \mathcal{X})}  \left \{  \bx^\intercal \left( \bbu  - \eta \bbl \right) {\btheta}  \right \} \le \eta \lambda - \mu\\
&\iff \max_{\bx \in \R^n: \bba \bx \le \bb}  \left \{  \bx^\intercal \left( \bbu  - \eta \bbl \right) {\btheta}  \right \} \le \eta \lambda - \mu\\
&\iff \exists \bgamma \ge \bzero \text{ such that }  \bb^\intercal \bgamma  \le \eta \lambda - \mu \text{ and }\bba^\intercal \bgamma = \left(  \bbu  - \eta \bbl \right) {\btheta}  
\end{align*}
where the first line follows from algebra and the assumptions that $\bx^\intercal \bbl {\btheta} + \lambda > 0$ for all $\bx \in \mathcal{X}$ and $\btheta \in \mathcal{U}$, the second line is the fundamental theorem of linear programming, the third line follows from the earlier assumption that $\conv(\mathcal{X}) = \{ \bx \in \R^n: \bba \bx \le \bb \}$, and the fourth line follows from strong duality. The above implies that the reduced uncertainty set admits a polynomial-size extended formulation, i.e., 
\begin{align*}
    \mathcal{U}_\eta = \left \{ \btheta \in \R^p:\;\; \begin{aligned}
        &\bbd \btheta \ge \bg \\
&\max_{\bx \in \mathcal{X}} \frac{\bx^\intercal \bbu {\btheta} + \mu}{\bx^\intercal \bbl {\btheta} + \lambda} \le \eta    \end{aligned} \right \} 
        =  \left \{\btheta \in \R^p: \begin{aligned}
&\bbd \btheta \ge \bg \\
&\exists \bgamma \ge \bzero \text{ such that } \bb^\intercal \bgamma  \le \eta \lambda - \mu \text{ and }\bba^\intercal \bgamma = \left(  \bbu  - \eta \bbl \right) {\btheta} 
\end{aligned} \right \}.
\end{align*}
We observe that
\begin{align}
\max_{\bx \in \mathcal{X}}\min_{\btheta \in \mathcal{U}_\eta}  \frac{\bx^\intercal \bbu {\btheta} + \mu}{\bx^\intercal \bbl {\btheta} + \lambda} &=\left [ \begin{aligned}
&\underset{\bx \in \mathcal{X}, t \in \R}{\text{maximize}}&& t \\
&\text{subject to}&& t \le \min_{\btheta \in \mathcal{U}_\eta}  \frac{\bx^\intercal \bbu {\btheta} + \mu}{\bx^\intercal \bbl {\btheta} + \lambda} 
\end{aligned} \right] \notag \\
&=\left [ \begin{aligned}
&\underset{\bx \in \mathcal{X}, t \in \R}{\text{maximize}}&& t \\
&\text{subject to}&& t \lambda - \mu  \le \min_{\btheta \in \mathcal{U}_\eta} \bx^\intercal\left( \bbu - t \bbl \right) {\btheta} 
\end{aligned} \right]  \label{line:messy_outer}
\end{align}
Using the polyhedral representation of $\mathcal{U}_\eta$ developed above, we observe for each $\bx \in \mathcal{X}$ and $t \in \R$ that 
\begin{align}
\min_{\btheta \in \mathcal{U}_\eta} \bx^\intercal\left( \bbu - t \bbl \right) {\btheta}  = \left[ \begin{aligned}
&\underset{\bpsi \in \R^q, \by \in \R^n, \sigma \in \R}{\text{maximize}} &&    \bg^\intercal \bpsi - (\eta \lambda - \mu) \sigma\\
&\text{subject to}&&
      \bbd^\intercal \bpsi - \left( \bbu - \eta \bbl \right)^\intercal \by
        = \left( \bbu - t \bbl \right)^\intercal \bx, \\
 &&&     \bba \by \le \bb \sigma\\
 &&&  \bpsi \ge \bzero,\sigma \ge 0
    \end{aligned}\right] \label{line:messy_internal_dual}
\end{align}
Combining \eqref{line:messy_outer} and \eqref{line:messy_internal_dual} and letting $\bz = t \bx$ denote the bilinear term, we obtain the desired result.
\end{proof}

\begin{proof}[Proof of Proposition~\ref{prop:reform:simple}]
In the special case of $n = p$, $\bbl = \bzero$, $\bbu$ is the identity matrix, $\lambda = 1$, and $\mu = 0$, the optimization problem from Lemma~\ref{lem:reformulation} becomes
\begin{align*}
\left[ \begin{aligned}
&\underset{ \bx \in \{0,1\}^n, \bpsi \in \R^q, \by \in \R^n, \sigma \in \R}{\text{maximize}} &&   \bg^\intercal \bpsi - \eta \sigma\\
&\text{subject to}&&  \bbd^\intercal \bpsi -  \by
        = \bx\\
        &&& \bba \bx \le \bb\\
 &&&     \bba \by \le \bb \sigma\\
 &&&  \bpsi \ge \bzero,\sigma \ge 0
    \end{aligned}\right] 
    \end{align*}
    Substituting $\by = \bbd^\intercal \bpsi - \bx$ and rearranging terms,  we obtain the desired reformulation in Proposition~\ref{prop:reform:simple}. 
    \end{proof}
    
    \begin{proof}[Proof of Proposition~\ref{prop:reform:general}]
    If the optimal value of \eqref{prob:fraction_eta} is in $[\ubar{t},\bar{t}]$ for $0 \le \ubar{t} \le \bar{t}$, then we observe that the bilinear term 
    \begin{align*}
    \bz = t \bx
    \end{align*}
    in the reformulation from Lemma~\ref{lem:reformulation} can be replaced with the McCormick envelopes 
    \begin{alignat*}{2}
    \ubar{t} & \le t \le \bar{t}\\
    \ubar{t} x_i &\le z_i \le \bar{t}x_i &\quad& \forall i \in \{1,\ldots,n\}\\
    t - \bar{t} (1-x_i) &\le z_i   \le t - \ubar{t} (1-x_i) &\quad& \forall i \in \{1,\ldots,n\}
    \end{alignat*}
    Substituting in those constraints, we obtain the desired reformulation in Proposition~\ref{prop:reform:general}. 
    \end{proof}

\end{document}